\documentclass{amsart}%
\usepackage{amssymb}
\usepackage{color}
\usepackage{amsmath}
\usepackage{graphicx}
\usepackage{amsfonts}%
\newcommand{\NN}{\ensuremath{\mathbb{N}}}

\newcommand{\PP}{\ensuremath{\mathbb{P}}}

\newcommand{\RR}{\ensuremath{\mathbb{R}}}

\renewcommand{\phi}{\varphi}

\newcommand{\eps}{\ensuremath{\epsilon}}

\newcommand{\bB}{\ensuremath{\mathcal{B}}}

\newcommand{\dD}{\ensuremath{\mathcal{D}}}

\newcommand{\fF}{\ensuremath{\mathcal{F}}}

\newcommand{\oO}{\ensuremath{\mathcal{O}}}

\newcommand{\tT}{\ensuremath{\mathcal{T}}}

\newtheorem{theorem}{Theorem}

\newtheorem{corollary}[theorem]{Corollary}

\newtheorem{definition}[theorem]{Definition}

\newtheorem{lemma}[theorem]{Lemma}

\newtheorem{remark}[theorem]{Remark}

\title[RDS for SEEs with fractional noise]
{Random dynamical systems for stochastic evolution equations driven by multiplicative fractional Brownian noise with Hurst parameters $H\in (1/3,1/2]$}

\keywords{Stochastic PDEs, fractional Brownian motion, pathwise solutions, rough path theory. \\
This work was partially supported by MTM2011-22411, FEDER founding (M.J.
Garrido-Atienza and B. Schmalfu{\ss}), and by NSF0909400 (K. Lu). }

\author{Mar\'{\i}a J. Garrido-Atienza}
\address[Mar\'{\i}a J. Garrido-Atienza]{Dpto. Ecuaciones Diferenciales y An\'alisis Num\'erico\\
Universidad de Sevilla, Apdo. de Correos 1160, 41080-Sevilla,
Spain} \email[Mar\'{\i}a J. Garrido-Atienza]{mgarrido@us.es}

\author{Kening Lu}
\address[Kening Lu]{346 TMCB\\
Brigham Young University, Provo, UT 84602, USA} \email[Kening
Lu]{klu@math.byu.edu}

\author{Bj{\"o}rn Schmalfu{\ss }}
\address[Bj{\"o}rn Schmalfu{\ss }]{Institut f\"{u}r Stochastik\\
Friedrich Schiller Universit{\"a}t Jena, Ernst Abbe Platz 2, 77043\\
Jena,
Germany\\
 }
\email[Bj{\"o}rn Schmalfu{\ss }]{bjoern.schmalfuss@uni-jena.de}

\keywords{Stochastic PDEs, Hilbert-valued fractional Brownian motion, pathwise solutions. \\
This work was partially supported by MTM2011-22411, FEDER founding (M.J.
Garrido-Atienza and B. Schmalfu\ss), and by NSF0909400 (K. Lu). }

\begin{document}

\begin{abstract}
We consider the stochastic evolution equation
\begin{equation*}
  du=Audt+G(u)d\omega,\quad u(0)=u_0
\end{equation*}
in a separable Hilbert--space $V$. Here $G$ is supposed to be three times Fr\'echet--differentiable and $\omega$ is a trace class fractional Brownian--motion with Hurst parameter $H\in (1/3,1/2]$. We prove the existence of a global solution where exceptional sets are independent of the initial state $u_0\in V$.
In addition, we show that the above equation generates a random dynamical system.

\end{abstract}

\maketitle

\section*{\today}
\section{Introduction}\label{s0}
This paper can be seen as the second part of an outgoing project whose main aim is to prove that stochastic evolution equations (SEEs)
\begin{equation}\label{grds_eq01}
  du=Audt+G(u)\,d\omega,\qquad u(0)=u_0
\end{equation}
 on a separable Hilbert--space $V$ generate random dynamical systems. Our interpretation of the solution will be given in a mild form
\begin{equation}\label{grds_eq011}
  u(t)=S(t)u_0+\int_0^tS(t-r)G(u(r))d\omega(r)
\end{equation}
assuming that $A$ is a negative operator that generates a semigroup $S$ and $G$ is sufficiently regular.
Our purpose is to study this problem if $\omega$ is H{\"o}lder--continuous on compact intervals $[0,T]$. If $\omega$ is a regular trajectory, say
$\omega^\prime \in L_\infty(0,T,V)$, then a classical formulation of the above problem is given by
\begin{equation}\label{grds_eq10}
   u(t)=S(t)u_0+\int_0^tS(t-r)G(u(r))\omega^\prime(r)dr.
\end{equation}
If $S$ is an analytic semigroup, under typical assumptions on $G$ there exists a unique solution $u\in C_{\gamma}([0,T];V)$ if $u_0\in D((-A))$, where $\gamma\in (0,1)$, whereas $u\in C_{\gamma,\sim}([0,T];V)$ if $u_0\in V$. This latter space is a suitable modification of the space of $\gamma$--H\"older--continuous functions and will be introduced in Section 2.

\smallskip

So far, the existence of a {\it local} mild solution to (\ref{grds_eq011}) has been obtained in \cite{GaLuSch}, assuming that $u_0\in V$. To be more precise, this solution is derived as the path component of a path--area solution pair $(u,v)$ of a system that considers not only the original equation but also a second equation, a natural extension of the area object $(u\otimes \omega)$. This pair $(u,v)$ is related to the noisy path $\omega$ by the well-known Chen property. The method presented in \cite{GaLuSch} can be considered as a generalization of the one developed by Hu and Nualart \cite{HuNu09} in the finite--dimensional setting, but with some important differences as that in the infinite-dimensional framework it is necessary to construct an area element, denoted by $(\omega\otimes_S \omega)$, depending on $\omega$ as well as the semigroup $S$, that satisfies nice properties. The considered (pathwise) integral in \cite{GaLuSch} was also previously introduced in \cite{HuNu09} for ordinary differential equation with H{\"o}lder--continuous noise term. The reasons for which we have chosen such a pathwise integral will be detailed below.  However, another definition to integrate against the fractional Brownian motion would be to consider the Rough Path Theory, see for instance \cite{DGT}, \cite{DeNeTi10} and the references therein.
\smallskip

In the current paper we want to go one step further and establish the existence of a global solution to (\ref{grds_eq011}). In comparison to \cite{GaLuSch}, we here replace the regular initial condition of that paper by a less regular one, namely $u_0\in D((-A)^\kappa)$ for suitable $\kappa>0$, and make a slightly modification of the phase spaces, but that still ensures the additivity of the pathwise integral. Then, under additional regularity conditions on $G$ we prove the existence of a global solution to (\ref{grds_eq011}).  Let us mention that the estimates of the different integrals defining the path--area formulation of our problem are quadratic, and therefore, in a first step we are only able to establish the existence of a local solution $(u^1,v^1)$ such that the path component $u^1$ is defined on a time interval $[T_0,T_1]$. However, since in particular $u^1(T_1)$ also belongs to $D((-A)^\kappa)$, we can pick it as a new initial condition, and hence we get a new solution $(u^2,v^2)$ such that $u^2$ is now defined on a time interval $[T_1,T_2]$. Proceeding in a similar way, after a finite number of these steps one finally derives the existence of a global solution on $[T_0,T]$ for a given $T>0$ provided that $u_0\in D((-A)^\kappa)$. The way in which we finally can establish the existence of a global solution is based in a concatenation procedure. Furthermore, this method also provides the existence of a global solution when $u_0 \in V$, since roughly speaking what one has to do is to concatenate the local solution obtained in \cite{GaLuSch} with the aforementioned global one.

\smallskip

Once the existence of a global solution is solved, it is quite natural to study whether that solution generates a random dynamical, which turns out to be a powerful property to analyze the asymptotic behavior of the solution by using all the machinery of the random dynamical systems theory. The reader is referred to \cite{Arn98} for a complete description of that theory.

The fact that an ordinary It\^o--equation
\begin{equation*}
  du=G(u)d\omega,\,u(0)=u_0\in \RR^d
\end{equation*}
where $\omega$ here is a finite dimensional Brownian motion,
generates a random dynamical system is due to the flow property,  see Kunita \cite{Kun90}.
To obtain this flow property one needs to apply Kolmogorov's theorem about the existence of a continuous random field
with finite--dimensional parameter range. Scheutzow in \cite{Sche96} derived the existence of a random dynamical system from this flow property.

\smallskip
Results on the existence and uniqueness for It\^o--SEEs having as a state space an infinite dimensional separable Banach--space are established in Da Prato and Zabczyk \cite{DaPrato}. The point is that, unfortunately, it is not known how to obtain stochastic flows for these It\^o--SEEs, and the main reason is because it is not known how to extend Kolmogorov's theorem to an infinite dimensional parameter range, which would be appropriate for dealing with SEEs. More precisely, solutions of SEEs are defined almost surely where the exceptional sets depend on the initial condition. But it is complicated to generate a random dynamical system if more than countable many exceptional sets may appear. Nevertheless, there are results getting the random dynamical system for SEEs with very special noise terms, either SEEs driven by additive noise or linear multiplicative noise. These special noises make it possible to transform such  It\^o--SEE into a pathwise evolution equation which is appropriate to generate a random dynamical system.

\smallskip

The ansatz in this paper is quite different. Instead of considering the usual stochastic integral, for instance in the It\^o sense, as we have previously mentioned we consider a pathwise integral
which is well defined for {\em any} $\beta^\prime$--H{\"o}lder--continuous integrator ($1/3<\beta^\prime$) if the integrand is sufficiently regular. As we will explain below, this regularity can be described in terms of a modified space of H{\"o}lder--continuous functions. Furthermore, by the choice of that integral the unique solution has pathwise character, which means we can avoid exceptional sets depending on the initial states and making possible to investigate whether it generates a random dynamical system.

\smallskip
The article is organized as follows. In Section 2 we collect tools from functional analysis that will be applied later. In particular, we introduce analytic semigroups, special non-linear operators, tensor products and function spaces given by modification of H{\"o}lder--continuous functions. In Section 3 we at first define so called fractional derivatives allowing us define an integral with H{\"o}lder--continuous integrator. In addition we give the main properties of such an integral. In Section 4  we define the mild--path area solution of an SEE. In particular, we formulate an operator equation whose fixed point
presents this kind of (local) solution, and mention an approximation result. In Section 5 we investigate how the local solution can be extended to any interval $[0,T]$ by means of the concatenation method. For this type of solution we can avoid state dependent exceptional sets which then is the key to prove in Section 6 the existence of a random dynamical system generated by an SEE. Finally in Section 7 a stochastic partial differential equation is discussed as an example for an SEE generating a mild path--area solution.

\section{Preliminaries}\label{s1}
In the following we denote by $V,\,\hat V$ or $\tilde V$ separable Hilbert--spaces. Then as usual $L(V,\hat V)$ denotes the Banach--space of linear operators from $V$ to $\hat V$ and $L_2(V,\hat V)$ is the separable Hilbert--space of Hilbert--Schmidt--operators from $V$ to $\hat V$.

\smallskip

Let $S$ be an analytic semigroup on $V$ with generator $A$. $-A$ is supposed to be symmetric, positive and have an inverse which is compact. Then $-A$ has a discrete spectrum  $0<\lambda_1\le\lambda_2\le\cdots$
of finite multiplicity tending to $+\infty$. The associated eigenelements $(e_i)_{i\in\NN}$ are written such that they form a complete orthonormal system of $V$.
In addition, we can define the associated fractional powers of $-A$ denoted by $(-A)^\delta,\,\delta\in\RR$
with norm  $|x|_{D((-A)^{\delta})}:=|(-A)^{\delta}x|$, see Pazy \cite{Pazy} Section 2.6. We denote $V_\delta=D((-A)^\delta)$.

Collecting some properties of $S$ we have that
\begin{lemma}\label{grds_l1}
If $T>0$, $0\le \delta \leq \gamma$, $\sigma-\theta\in [0,1]$, then there exists a $c>0$ such that for $t\in [0,T]$
\begin{align}\label{eq1}
  &\|S(t)\|_{L(V_\delta, V_\gamma)}=  \|(-A)^\gamma S(t)\|_{L(V_\delta,V)}\le ct^{\delta-\gamma},
\end{align}
\begin{align}\label{eq2}
 &\|S(t)-{\rm
id}\|_{L(V_{\sigma},V_{\theta})} \le c
t^{\sigma-\theta}.
\end{align}

\smallskip

From these two properties, one can easily deduce that for any $\nu, \mu \in [0,1]$, $0\le\delta \leq \gamma+\nu$, $0\le \rho$ and $T>0$ there exists a constant $c>0$ such that for $0 \leq q<r< s<t\le T$ we have that
\begin{align*}
\|S(&t-r)-S(t-q)\|_{L(V_{\delta},V_{\gamma})}\le c(r-q)^\nu(t-r)^{-\nu-\gamma+\delta},\\
\|S(&t-r)- S(s-r)-S(t-q)+S(s-q)\|_{L(V_{\rho})}\\&\leq
c(t-s)^{\mu}(r-q)^{\nu}(s-r)^{-(\nu+\mu)}.
\end{align*}
\end{lemma}

Denote by $V\otimes V$ the separable Hilbert--space of tensor products of $V$, see Kadinson and Ringrose \cite{KadRin97}.
Moreover, for $x,\,y\in V$ the element $(x\otimes_V y)$ denotes the rank--one tensor product. In particular,
$((e_i\otimes_V e_j))_{i,j\in\NN}$ is a complete orthonormal system for $V\otimes V$
if $(e_i)_{i\in\NN}$ is a complete  orthonormal system  of $V$.

Consider a bilinear operator $B\in L_2(V\times V,\hat V)$ such that
\begin{equation}\label{grds_eq1}
  \sum_{i,j}|B(e_i,e_j)|_{\hat V}^2<\infty.
\end{equation}
Then it is possible to extend $B$ to a linear continuous operator $\hat B\in L_2(V\otimes V,\hat V)$.
More precisely, there exists a weak Hilbert--Schmidt--mapping $p:V\times V\to V\otimes V$ such that
$p(e_i,e_j)=(e_i\otimes_V e_j)$ and $\hat B$ is determined by factorization: $B=\hat B p$, and
\begin{equation*}
   \|\hat B\|_{L_2(V\otimes V,\hat V)}^2:= \sum_{i,j}|\hat B(e_i\otimes_V e_j)|_{\hat V}^2=\sum_{i,j}|B(e_i,e_j)|_{\hat V}^2= \|B\|_{L_2(V\times V,\hat V)}^2.
\end{equation*}
For these properties we refer to \cite{KadRin97}.
In the following we will write for  $\hat B$  the symbol of the original $B$.

\smallskip

Let $\hat V$ be a subspace of $V$. Consider now $G$ to be a Fre\'chet-differentiable mapping
\begin{equation*}
  V\ni u\mapsto G(u)\in L_2(V,\hat V)
\end{equation*}
with derivative
\begin{equation*}
  V\ni u\mapsto DG(u)\in L_2(V,L_2(V,\hat V))\cong L_2(V\times V,\hat V),
\end{equation*}
and therefore $DG(u)$ can be interpreted as an element of $L_2(V\otimes V,\hat V)$.\\

In what follows $|\cdot|$ represents the norm of $V$.
\smallskip

Let us formulate some  $G$.

\begin{lemma}\label{locall2}
Assume that the mapping $G: V\to L_2(V,\hat V )$ is bounded and three times continuously
Fr\'echet--differentiable with bounded first, second and third
derivatives $DG(u)$, $D^2G(u)$ and $D^3G(u)$, for $u\in V$. Let us denote, respectively, by $c_G$, $c_{DG},\, c_{D^2G}$ and $c_{D^3G}$
the bounds for $G$, $DG$, $D^2G$ and $D^3G$.
Then, for $u_1,\,u_2,\,v_1,\,v_2\in V$, we have
\begin{itemize}
\item $\|G(u_1)\|_{L_2(V,\hat V)}\le c_G$,\\
\item $\|G(u_1)-G(v_1)\|_{L_2(V, \hat V)}\le c_{DG}|u_1-v_1|$,\\
\item  $\|DG(u_1)-DG(v_1)\|_{L_2{(V\times V,\hat V)}}\le c_{D^2G}|u_1-v_1|$,\\
\item  $\|G(u_1)-G(u_2)-DG(u_2)(u_1-u_2)\|_{L_2(V, \hat V)}\le c_{D^2G}|u_1-u_2|^2$,\\
\item  $\|G(u_1)-G(v_1)-(G(u_2)-G(v_2))\|_{L_2(V, \hat V)}\le c_{DG}|u_1-v_1-(u_2-v_2)|\\  \qquad +c_{D^2G} |u_1-u_2|(|u_1-v_1|+|u_2-v_2|)$,\\
 \item $\|DG(u_1)-DG(v_1)-(DG(u_2)-DG(v_2))\|_{L_2(V\times V,\hat V)}\\
 \qquad  \le  c_{D^2G}|u_1-v_1-(u_2-v_2)|+c_{D^3G} |u_1-u_2|(|u_1-v_1|+|u_2-v_2|),$\\
\item   $\|G(u_1)-G(u_2)-DG(u_2)(u_1-u_2)-(G(v_1)-G(v_2)-DG(v_2)(v_1-v_2))\|_{L_2(V, \hat V)}\\
  \qquad \le c_{D^2G}
    (|u_1-u_2|+|v_1-v_2|)|u_1-v_1-(u_2-v_2)|\\
    +c_{D^3G}|v_1-v_2||u_2-v_2|(|u_1-u_2|+|u_1-v_1-(u_2-v_2)|).$
\end{itemize}
\end{lemma}
The proof of these properties is standard. We refer partially to \cite{MasNua03}, and for the last inequality to \cite{HuNu09}, page 2716.

\smallskip

Let us introduce some function spaces. For $\beta\in (0,1)$, denote by $C_\beta([0,T];V)$ the space of $\beta$-H{\"o}lder--continuous functions
with seminorm
\begin{equation*}
|||u|||_\beta=\sup_{0\leq s< t\leq  T}\frac{|u(t)-u(s)|}{(t-s)^\beta}.
\end{equation*}

In this space we could consider the usual norm
\begin{equation*}
\|u\|_{\beta}=\|u\|_C+|||u|||_\beta,\;  \text { with } \|u\|_C=\sup_{0\le s\le T}|u(s)|.
\end{equation*}
However, the above norm is  to be equivalent to the norm given by $|u(0)|+|||u|||_\beta$.
\\

A modification of the  space of H{\"o}lder--continuous functions is given by $C_{\beta,\sim}([0,T];V)$ with norm
\begin{equation*}
    ||u||_{\beta,\sim}=\|u\|_C+\sup_{0< s< t\leq  T}s^\beta\frac{|u(t)-u(s)|}{(t-s)^\beta}.
\end{equation*}

\begin{lemma}\label{grds_l2}
$C_{\beta,\sim}([0,T];V)$ is a Banach--space.
\end{lemma}

The proof can be found in Chen {\it et al.} \cite{ChGGSch12}.\\

\smallskip
For the following we suppose that $0<\beta<\beta^\prime<1/2$.
Let $\Delta_{0,T}$ be the triangle $\{(s,t)\in \RR^2:0< s\le t\le T\}$ and $\bar \Delta_{0,T}$ its closure in $\RR^2$.
Denote by
$C_{\beta+\beta^\prime}(\bar \Delta_{0,T};V\otimes V)$ the Banach--space of continuous functions defined on $\bar\Delta_{0,T}$, that are zero for $s=t$ and with norm
\begin{equation*}
   \|v\|_{\beta+\beta^\prime}=\sup_{0 \le s< t\leq  T} \frac{\|v(s,t)\|}{(t-s)^{\beta+\beta^\prime}}<\infty.
\end{equation*}

A modification of this space is given by
$C_{\beta+\beta^\prime,\sim}(\Delta_{0,T};V\otimes V)$,  consisting of continuous
functions $v$ defined on $\Delta_{0,T}$, that are zero for $0<s=t$, such that
\begin{equation*}
   \|v\|_{\beta+\beta^\prime,\sim}=\sup_{0< s< t\leq  T} s^{\beta}\frac{\|v(s,t)\|}{(t-s)^{\beta+\beta^\prime}}<\infty.
   \end{equation*}
These functions may not be defined for $s=0$ and may have a singularity for $(s,t),\,s=0$.
Replacing $V\otimes V$ by $\RR$ an example for an element of this space is $(s,t)\mapsto s^{-\beta}(t-s)^{\beta+\beta^\prime}$.
\begin{lemma}\label{grdsl2}
The spaces $C_{\beta+\beta^\prime}(\bar \Delta_{0,T};V\otimes V)$, $C_{\beta+\beta^\prime,\sim}(\Delta_{0,T};V\otimes V)$ are Banach--spaces.
\end{lemma}
\begin{proof}
We consider only the second space.
Let $(K_n)_{n\in\NN}$, $K_n=\bar\Delta_{n^{-1},T}$ be an increasing sequence of compact subsets of $\Delta_{0,T}$ such that $\bigcup_n K_n=\Delta_{0,T}$ and let $(v_m)_{m\in\NN}$ be a Cauchy--sequence on
$C_{\beta+\beta^\prime,\sim}(\Delta_{0,T};V\otimes V)$. We then have
\begin{equation*}
  \sup_{\footnotesize
  \begin{array}{c}(s,t)\in K_n\\s<t\end{array}}s^{\beta}\frac{\|v_m(s,t)-v_{m^\prime}(s,t)\|}{(t-s)^{\beta+\beta^\prime}}\ge \frac{1}{T^{\beta+\beta^\prime}n^\beta} \sup_{\footnotesize
  \begin{array}{c}(s,t)\in K_n\\s<t\end{array}}\|v_m(s,t)-v_{m^\prime}(s,t)\|.
\end{equation*}
Hence for any $n\in\NN$ we have that $(v_m)_{m\in\NN}$ is a  Cauchy--sequence in  $K_n$ such that there is a uniform limit $v$  which is continuous on any $K_n$ and thus continuous on $\Delta_{0,T}$. In addition,
$v(s,s)=0$ for $0<s\le T$.

Furthermore, there exists a $c>0$ such that for all $m\in\NN$ and $(s,t)\in \Delta_{0,T},\,s<t$,
\begin{equation*}
  s^{\beta}\frac{\|v_m(s,t)\|}{(t-s)^{\beta+\beta^\prime}}\le c
\end{equation*}
and by the convergence of $v_m(s,t)$ to $v(s,t)$ we obtain
\begin{equation*}
  s^{\beta}\frac{\|v(s,t)\|}{(t-s)^{\beta+\beta^\prime}}\le c
\end{equation*}
such that $v\in C_{\beta+\beta^\prime,\sim}(\Delta_{0,T};V\otimes V)$.
Now for any  $\eps>0$ and for $m^\prime>m\ge N(\eps)$ and $(s,t)\in \Delta_{0,T}, s<t$, we get
\begin{equation*}
s^{\beta}\frac{\|v_m(s,t)-v_{m^\prime}(s,t)\|}{(t-s)^{\beta+\beta^\prime}}<\eps
\end{equation*}
and thus for $m^\prime\to\infty$
\begin{equation*}
s^{\beta}\frac{\|v_m(s,t)-v(s,t)\|}{(t-s)^{\beta+\beta^\prime}}\le\eps
\end{equation*}
which shows the convergence of $(v_m)_{m\in\NN}$ to $v$ in $C_{\beta+\beta^\prime,\sim}(\Delta_{0,T};V\otimes V)$.

\end{proof}

\section{Fractional pathwise integrals}\label{s2}
In this section we introduce a $V$-valued integral where the integrator is  not of bounded variation but H{\"o}lder--continuous. This guides us an infinite--dimensional version of the Young--integral. For this purpose we introduce $V$--valued fractional derivatives.

\smallskip

Let $F,\,\xi$ be sufficiently regular functions on $[s,t]$.
For $\alpha\in (0,1)$ we define the right--sided fractional derivative of order $\alpha$ of $F$ and the left--sided fractional derivative of order $1-\alpha$ of $\xi_{t-}(\cdot):=\xi(\cdot)-\xi(t)$, given for $0\le s\le r\le t$ by the expressions
 \begin{align*}
 \begin{split}
    D_{{s}+}^\alpha F[r]=&\frac{1}{\Gamma(1-\alpha)}\bigg(\frac{F(r)}{(r-s)^\alpha}+\alpha\int_{s}^r\frac{F(r)-F(q)}{(r-q)^{1+\alpha}}dq\bigg),\\
    D_{{t}-}^{1-\alpha} \xi_{{t}-}[r]=&\frac{(-1)^{1-\alpha}}{\Gamma(\alpha)}
    \bigg(\frac{\xi(r)-\xi(t)}{(t-r)^{1-\alpha}}+(1-\alpha)\int_r^{t}\frac{\xi(r)-\xi(q)}{(q-r)^{2-\alpha}}dq\bigg).
\end{split}
\end{align*}
Here $\Gamma(\cdot)$ denotes the Gamma function.

The following theorem states regularity conditions on $F,\,\xi$ such that the Young--integral exists.

\begin{theorem}\label{grds_t1}
Assume $\beta>\alpha$ and $\alpha+\beta^\prime >1$. Let $\tilde V,\,\hat V$ be two separable Hilbert spaces with complete orthonormal bases  $(\tilde e_i)_{i\in\NN}$ and $(\hat e_j)_{j\in\NN}$, respectively. Let
 $F\in C_{\beta,\sim}([0,T]; L_2(\tilde V,\hat V))$, $\xi \in C_{\beta^\prime} ([0,T]; \tilde V)$ and assume
\begin{equation*}
  r\mapsto\|D_{s+}^\alpha F[r]\|_{L_2(\tilde V,\hat V)}|D_{t-}^{1-\alpha}\xi_{t-}[r]|_{\tilde V}
\end{equation*}
is Lebesgue-integrable. Then for $0\le s\le r\le t \le T$ we  define
\begin{equation*}
  \int_s^t F(r)d\xi(r):=(-1)^\alpha\sum_{j}\bigg(\sum_i\int_s^t D_{s+}^\alpha(\hat e_j,F(\cdot)\tilde e_i)_{\hat V}[r]D_{t-}^{1-\alpha}(\tilde e_i,\xi_{t-}(\cdot))_{\tilde V}[r]dr\bigg)\hat e_j.
\end{equation*}
As a consequence,
\begin{equation*}
   \bigg|\int_s^t F(r)d\xi(r)\bigg|_{\hat V}\le\int_s^t \|D_{s+}^\alpha F[r]\|_{L_2(\tilde V,\hat V)}|D_{t-}^{1-\alpha}\xi_{t-}[r]|_{\tilde V}dr<\infty.
\end{equation*}
\end{theorem}
For the proof we refer to \cite{GaLuSch}. In particular, we have reduced the definition of this Hilbert--space valued integral to an infinite  sum of one dimensional Young--integrals.
\begin{corollary}\label{grds_c1}
Under the assumptions of Theorem \ref{grds_t1}, for any $T>0$ there exists a $c>0$ such that for any
$F\in C_{\beta,\sim}([0,T]; L_2(\tilde V,\hat V))$, $\xi \in C_{\beta^\prime} ([0,T]; \tilde V)$ and $0\le s<t\le T$
\begin{equation*}
   \bigg|\int_s^t  F(r)\xi(r)dr\bigg|_{\hat V}\le c(t-s)^{\beta^\prime}|||\xi|||_{\beta^\prime}||F||_{\beta,\sim}<\infty.
\end{equation*}
Similarly, if $F\in C_{\beta}([0,T]; L_2(\tilde V,\hat V))$, $\xi \in C_{\beta^\prime} ([0,T]; \tilde V)$, then we get
\begin{equation*}
   \bigg|\int_s^t  F(r)\xi(r)dr\bigg|_{\hat V}\le c(t-s)^{\beta^\prime}|||\xi|||_{\beta^\prime}(\|F\|_C+(t-s)^\beta|||F|||_\beta).
\end{equation*}
\end{corollary}

Note that the first estimate is an immediate consequence of the fact that if $F\in C_{\beta,\sim}([0,T];L_2(\tilde V,\hat V))$ with $\alpha+\beta>1$, and $\xi \in C_{\beta^\prime} ([0,T]; \tilde V)$ then
\begin{align*}
  \|D_{s+}^\alpha F[r]\|_{L_2(\tilde V,\hat V)}& \le c\frac{||F||_{\beta,\sim}}{(r-s)^\alpha},\\|D_{t-}^{1-\alpha} \xi_{t-}[r]|_{\tilde V}& \leq c |||\xi|||_{\beta^\prime} (t-r)^{\beta^\prime+\alpha-1}.
\end{align*}
The second inequality can be deduced in a similar way.\\

If now we suppose that $\beta \leq \alpha$ then the assumptions of Theorem \ref{grds_t1} does not make sense. To overcome this problem
we have to consider a modification of the fractional derivatives introduced above. More precisely, we introduce a new fractional derivative for a special function $F$ given by the integrand of our original problem \eqref{grds_eq011}, so that we set $F(t)=G(u(t))$
where operator $G$ satisfies the assumptions described in Lemma \ref{locall2}. In addition we set $\xi=\omega \in C_{\beta^\prime}([0,T];V)$ that will also satisfy the hypothesis {\bf (H3)} of Section \ref{slocal}.

Consider the compensated fractional derivative defined as
\begin{align*}
  \hat D_{s+}^\alpha G(u(\cdot))[r]=&\frac{1}{\Gamma(1-\alpha)}\bigg(\frac{G(u(r))}{(r-s)^\alpha}\\&+\alpha\int_s^r\frac{G(u(r))-G(u(q))-DG(u(q))(u(r)-u(q))}{(r-q)^{1+\alpha}}dq\bigg).
\end{align*}
For tensor valued elements $v:\Delta_{0,T}\to V\otimes V$  we define
\begin{align*}
\mathcal{D} _{{t}-}^{1-\alpha}v[r] =&\frac{(-1)^{1-\alpha}}{%
\Gamma(\alpha)} \bigg( \frac{v(r,t)}{(t-r)^{1-\alpha}}%
+(1-\alpha) \int_r^{t} \frac{v(r,q)}{(q-r)^{2-\alpha}}dq %
\bigg).
\end{align*}
Below we clarify under which conditions these fractional derivatives are well defined. With these fractional derivatives in mind, if the pair $(u,v)$ satisfies the so-called Chen--equality, that is, for $0<s\le r\le t$
\begin{equation}\label{chen}
   v(s,r)+v(r,t)+(u(r)-u(s))\otimes_V(\omega(t)-\omega(r))=
   v(s,t),
\end{equation}
then the definition of the (pathwise stochastic) integral for the current work is the following:
\begin{align}\label{localeq6}
\begin{split}
\int_s^tG(u)d\omega:
  =&(-1)^\alpha\int_s^t\hat D_{s+}^\alpha G(u(\cdot))[r]D_{t-}^{1-\alpha}\omega_{t-}[r]dr\\&-(-1)^{2\alpha-1}\int_s^tD_{s+}^{2\alpha-1} DG(u(\cdot))[r]
  D_{t-}^{1-\alpha}\dD_{t-}^{1-\alpha}v(\cdot,t)[r]dr.
\end{split}
\end{align}
Note that this definition is an infinite dimensional generalization of the concept of integral given in \cite{HuNu09}.

According to the previous definition, for the integrand of the integral on the left hand side we should rather write the term $G_v(u)$ instead of $G(u)$, but for the sake of simplification we keep $G(u)$. Note that, however, for sufficiently regular $\omega$
this integral can be rewritten in the sense of Theorem \ref{grds_t1} with $F(t)=G(u(t))$ if we replace $v(s,t)$ by
\begin{equation}\label{grds_eq3b}
(u\otimes \omega)(s,t)=\int_s^t(u(r)-u(s))\otimes_Vd\omega(r),
\end{equation}
see also the beginning of the next section.\\

For the existence of \eqref{localeq6} we have the following result.

\begin{theorem}\label{locall8}
Consider $1/3<\beta<\beta^\prime<1/2$, choose $\beta<\alpha<2\beta$, $\alpha+\beta^\prime>1,\, \beta+1>2\alpha$ and assume that $G$ satisfies the assumptions of Lemma \ref{locall2}. Then for $T>0$ there exists a $c>0$ such that for $u\in C_{\beta,\sim}([0,T];V)$, $v \in C_{\beta+\beta^\prime,\sim}(\Delta_{0,T};V\otimes V)$,
 and $\omega\in C_{\beta^\prime}([0,T];V)$ such that (\ref{chen}) is fulfilled, for $0\le s\le t\le T$
\begin{equation*}
  \bigg|\int_s^t G(u)d\omega\bigg|_{\hat V}\le ct^{\beta^\prime-\beta}((1+\|u\|_{\beta,\sim}^2)|||\omega|||_{\beta^\prime}+(1+\|u\|_{\beta,\sim})\|v\|_{\beta+\beta^\prime,\sim})(t-s)^{\beta}.
\end{equation*}
Similarly, with the same choice of the parameters, if $\omega\in C_{\beta^\prime}([0,T];V)$, $u\in C_{\beta}([0,T];V)$ and $v \in C_{\beta+\beta^\prime}(\Delta_{0,T};V\otimes V)$ are coupled by (\ref{chen}), then
\begin{equation*}
\label{localeq_1a}
  \bigg|\int_s^t G(u)d\omega\bigg|_{\hat V}\le ct^{\beta^\prime-\beta}((1+t^{2\beta}|||u|||_{\beta}^2)|||\omega|||_{\beta^\prime}+(1+ t^{2\beta}|||u|||_{\beta}\|v\|_{\beta+\beta^\prime}+t^{2\beta}\|v\|_{\beta+\beta^\prime}^2))(t-s)^{\beta}.
\end{equation*}
\end{theorem}

Note that if $u\in C_{\beta,\sim}([0,T];V)$ then
\begin{equation*}
  \|\hat D_{s+}^\alpha G(u)[r]\|_{L_2(V,\hat V)}\le c\frac{c_G+c_{D^2G}\|u\|_{\beta,\sim}^2}{(r-s)^{\alpha}},
\end{equation*}
and if $u\in C_{\beta}([0,T];V)$ we get
\begin{equation*}
  \|\hat D_{s+}^\alpha G(u)[r]\|_{L_2(V,\hat V)}\le c\frac{c_G+c_{D^2G}(r-s)^{2\beta}\|u\|_{\beta}^2}{(r-s)^{\alpha}},
\end{equation*}
hence the first of the integrals defining (\ref{localeq6}) gives us a similar estimate than in Theorem \ref{grds_t1}. For the second integral we can use Theorem \ref{grds_t1} directly replacing
$\alpha$ by $2\alpha-1$, $G(u(r))$ by $DG(u(r))$ defined on $V\otimes V$ and $\xi(r)$ by $\dD_{t-}^{1-\alpha}v(\cdot,t)[r]$. In particular, taking the Chen--equality and the regularity into account, in the first case we have
\begin{equation*}
   \|D_{t-}^{1-\alpha}\mathcal{D} _{{t}-}^{1-\alpha}v[r]\|\le c(\|v\|_{\beta+\beta^\prime,\sim}+\|u\|_{\beta,\sim}|||\omega|||_{\beta^\prime})(r-s)^{-\beta}(t-r)^{\beta+\beta^\prime+2\alpha-2},
\end{equation*}
and in second case  that
\begin{align*}
    &\| D_{s+}^{2\alpha-1} DG(u)[r]\|_{L_2(V\otimes V,\hat V)}\le c\frac{c_{DG}+c_{D^2G}(r-s)^\beta|||u|||_\beta}{(r-s)^{2\alpha-1}}\\
   &\|D_{t-}^{1-\alpha}\mathcal{D} _{{t}-}^{1-\alpha}v[r]\|\le c(\|v\|_{\beta+\beta^\prime}+|||u|||_\beta|||\omega|||_{\beta^\prime})(t-r)^{\beta+\beta^\prime+2\alpha-2}.
\end{align*}
The details of the proof of this theorem can be found in \cite{GaLuSch}.\\

To finish this section, we want to emphasize that, in contrast to the classical Ito--integral which is defined as a limit in probability of Daboux sums with respect to the increments of a Brownian--motion, with the above definition there are no exceptional sets of probability which in general depend on the integrand. In this sense our integral is pathwise, allowing us to investigate whether the solution to (\ref{grds_eq01}) generates a random dynamical system, see Section \ref{RDSs}.

\section{Pathwise local solutions of SEEs} \label{slocal}
In this section we formulate a result about the local existence of solutions to \eqref{grds_eq01} for a noise-path $\omega$ which is in $C_{\beta^\prime}([0,T];V)$, $\beta^\prime\in (1/3,1/2)$. In particular we formulate conditions such that \eqref{grds_eq011} has a local mild solution when the pathwise (stochastic) integral is defined in the sense of \eqref{localeq6}. For the proofs of the results given in this section we refer to Garrido-Atienza {\it et al.} \cite{GaLuSch}.
These results generalize those of Hu and Nualart \cite{HuNu09} to the case of a state space which is now an infinite dimensional Hilbert--space.

\smallskip

If $\omega$ were (for instance) a trace class Brownian--motion, according to the theory of Da Prato and Zabczyk \cite{DaPrato} the equation (\ref{grds_eq011}) would have a unique mild solution provided that $G$ and $S$ satisfy weaker conditions than the ones formulated in the last section. The corresponding integrals for these equations would be defined in the Ito--sense, and this means that they would not be pathwise, which is an obstacle to generate a random dynamical system. In Maslowski and Nualart \cite{MasNua03} these type of equations are solved by fractional techniques for $\beta^\prime$--H{\"o}lder--continuous integrators with $\beta^\prime>1/2$, where the integrals have been defined similarly to Theorem \ref{grds_t1}, allowing to introduce a random dynamical system, whose existence has been investigated recently by Chen {\it et al.} \cite{ChGGSch12}. However, the integral as introduced in Theorem \ref{grds_t1} is not well defined  when the H{\"o}lder--exponent of the integrator is less than or equal to 1/2, which is the case that we aim to consider in this paper.

\smallskip

We interpret the solution to \eqref{grds_eq011} as the first component of an element $U=(u,v)$ that satisfies of the operation equation
\begin{align*}
\begin{split}
 & u=\tT_1(U),\\
  &\tT_1(U,\omega,u_0)(t)
  =
S(t)u_0+(-1)^\alpha\int_0^t\hat D_{0+}^\alpha (S(t-\cdot)G(u(\cdot)))[r]D_{t-}^{1-\alpha}\omega[r]dr\\
  &-(-1)^{2\alpha-1}
  \int_0^tD_{0+}^{2\alpha-1}(S(t-\cdot)DG(u(\cdot)))[r]D_{t-}^{1-\alpha}\dD_{t-}^{1-\alpha}v[r]dr.
\end{split}
\end{align*}
Note that to define $\tT_1$ we have applied the definition given in \eqref{localeq6}, which makes necessary to set up and solve an equation for the second component $v$ of $U$. In order to do that, we assume for a while that the driving path $\omega$ is smooth enough, giving us the opportunity to know what is the suitable equation to be satisfied for the $v$ component, to later make the adequate assumptions to translate the situation to the case of having a H\"older--continuous driving path $\omega$ with H\"older--exponent $\beta^\prime\in (1/3,1/2)$.  Therefore, suppose that $\omega$ is smooth, so that $v$ is given by (\ref{grds_eq3b}). For such a smooth $\omega$, combining (\ref{grds_eq011}) with (\ref{grds_eq3b}) we have
\begin{align*}
\begin{split}
    (u\otimes\omega)(s,t)=&\int_s^t(S(\xi-s)-{\rm id})u(s)\otimes_V\omega^\prime(\xi)d\xi\\
    &+\int_s^t\int_s^\xi S(\xi-r)G(u(r))\omega^\prime(r)dr\otimes_V\omega^\prime(\xi)d\xi.
    \end{split}
\end{align*}
Representing the appearing integrals in fractional sense we obtain
\begin{align}\label{grds_eq4}
\begin{split}
    (u\otimes\omega)&(s,t)
    =(-1)^\alpha\int_s^tD_{s+}^\alpha((S(\cdot-s)-{\rm id})u(s))[\xi]\otimes_VD_{t-}^{1-\alpha}\omega[\xi]d\xi\\
    &-(-1)^\alpha\int_s^t \hat D_{s+}^\alpha
    G(u(\cdot))[r] D_{t-}^{1-\alpha}(\omega\otimes_S\omega)(\cdot,t)_{t-}[r]dr\\
&+(-1)^{2\alpha-1}\int_s^t D_{s+}^{2\alpha-1} DG(u(\cdot))[r] D_{t-}^{1-\alpha}\dD_{t-}^{1-\alpha}(u\otimes(\omega\otimes_S\omega)(t))[r]dr.
\end{split}
\end{align}

For such smooth $\omega$, the operator $(\omega\otimes_S\omega)$ is given by
\begin{align}\begin{split}
  L_2(V,\hat V)\ni E\mapsto E(\omega\otimes_S\omega)(s,t)&=\int_s^t\int_s^\xi
    S(\xi-r)E\omega^\prime(r)dr \otimes_V \omega^\prime(\xi)d\xi\\
   &=\int_s^t\int_r^t
   S(\xi-r)E\omega^\prime(r)d\xi \otimes_V \omega^\prime(\xi)dr.
   \end{split}\label{omegaSS}
\end{align}

In particular, $(\omega\otimes_S\omega)$ satisfies the following Chen--equality
\begin{align*}
&E(\omega\otimes_S \omega)(s,r) +E(\omega\otimes_S
\omega)(r,t) +(-1)^{-\alpha} \omega_S(r,t)S_\omega(s,r)E\\
&\quad =E(\omega\otimes_S \omega)(s,t),
\end{align*}
where
\begin{align}
\label{localeq16}
\begin{split}
&e\in V\mapsto\omega_S(s,t)e=(-1)^{-\alpha}\int_s^tS(\xi-s)e\otimes_V\omega^\prime(\xi)d\xi,\\
&E\in L_2(V,\hat V) \mapsto S_\omega(s,t)E=  \int_s^t S(t-r)E\omega^\prime(r)dr.
\end{split}
\end{align}
The different expressions in (\ref{omegaSS}) and (\ref{localeq16}) are well-defined when $\omega$ is regular, see \cite{GaLuSch} for details.\\

Now let us come back to the original situation in which $\omega \in C_{\beta^\prime}([0,T];V)$ for $\beta^\prime\in (1/3,1/2)$. Denote by $\omega_n$ a piecewise linear (continuous) approximation of $\omega$ with respect to an equidistant partition of length $2^{-n}T$ such that $\omega_n(t)=\omega(t)$ for the partition points $t$. For these $\omega_n$ the operator
$(\omega_n\otimes_S\omega_n)$ can be defined according to (\ref{omegaSS}).\\

We stress that we have to give to $(\omega\otimes_S\omega)$ some meaning when $\omega$ is not smooth. To be precise, this meaning is formulated in the following Hypothesis {\bf (H)}, consisting in the following three assumptions:\\

{\bf (H1)} Let $H\in (1/3,1/2]$ and let  $1/3<\beta< \beta^\prime <H$. Suppose that there is an $\alpha$ such that $1-\beta< \alpha <2 \beta$,
$\alpha< \frac{\beta+1}{2}$.

\smallskip

{\bf (H2)} Let $A$ be the generator of an analytic semigroup $S$ and let $G: V\to L_2(V,\hat V)$
be a non--linear mapping satisfying the assumptions of Lemma \ref{locall2}, with $\hat V\subset V$ such that the embedding operator is Hilbert-Schmidt.

\smallskip

{\bf (H3)}
Let $(\omega_n)_{n\in\mathbb{N}}$ be a sequence  of piecewise linear functions with values in $V$ such that
$((\omega_n\otimes_S\omega_n))_{n\in\mathbb{N}}$ is defined by
\eqref{omegaSS}. Assume then that for any $\beta^\prime<H$
$$\lim_{n\to \infty} (\omega_n,(\omega_n\otimes_S\omega_n))=(\omega,(\omega\otimes_S\omega))\, \text {in }
C_{\beta^\prime}([0,T];V)\times C_{2\beta^\prime} (\bar \Delta_{0,T}; L_2(L_2(V,\hat V),V\otimes V)).$$

The parameter $H$ in the above assumption {\bf(H1)} will represent the Hurst parameter of a fractional Brownian motion in Section \ref{RDSs}.

In the recent paper Garrido-Atienza {\it et al.} \cite{GLS14d} the authors propose two different settings where assumption {\bf (H3)} is satisfied: the first one considers as driving noise a trace--class fractional Brownian--motion $B^H$ with $H\in (1/3,1/2]$ and values in a Hilbert--space, whereas by another less restrictive method, the second one considers an infinite-dimensional trace-class Brownian--motion $B^{1/2}$. \\

Now we should realize that the structure of \eqref{grds_eq4} is quite similar to the  structure of $\tT_1$ when replacing $\omega$ by $(\omega\otimes_S\omega)$ and
$v=(u\otimes\omega)$ by $(u\otimes(\omega\otimes_S\omega))$. Furthermore, if we write $w(t)=(u\otimes(\omega\otimes_S\omega)(t))$, then $w$ can be interpreted by fractional integrals as
\begin{align}\label{wholder}
\begin{split}
\tilde  E &w(t,s,q)=-\int_s^q\hat D_{s+}^{\alpha}\omega_S(\cdot,t)\tilde E(u(\cdot)-u(s),\cdot)[r]  D_{q-}^{1-\alpha}\omega_{q-}[r]dr\\
&+(-1)^{\alpha-1}\int_s^q D_{s+}^{2\alpha-1}\tilde E(u(\cdot)-u(s),\cdot)[r]  D_{q-}^{1-\alpha}\dD_{q-}^{1-\alpha}(\omega_S(t)\otimes\omega)[r]
    dr\\&+(-1)^{\alpha-1}\int_s^q D_{s+}^{2\alpha-1}\omega_S(\cdot,t)[r] \tilde E D_{q-}^{1-\alpha}\dD_{q-}^{1-\alpha}v[r] dr
\end{split}
\end{align}

where $\tilde E\in L_2(V\otimes V, \hat V)$. In addition $w$ satisfies a special form of the Chen--equality, see \cite{GaLuSch}.  The previous considerations make possible to formulate an equation for the second component $v$ of $U$ given by
\begin{equation*}
  v=\tT_2(U)
\end{equation*}
where the operator $\tT_2(U)$ is defined by (\ref{grds_eq4}) substituting $(u\otimes(\omega\otimes_S\omega))$ by $w$, and $w$ fulfills (\ref{wholder}).

In the sequel, we denote $\tT(U)=(\tT_1(U),\tT_2(U))$. Sometimes we will write
$$\tT(U,\omega,(\omega\otimes_S\omega), u_0)=(\tT_1(U,\omega, u_0),\tT_2(U,\omega,(\omega\otimes_S\omega), u_0))$$ to show the different ingredients in the operator. Furthermore, when this operator is defined on $[s,T]$, for $s>0$, instead of on $[0,T]$, later we will denote it by $\tT^s$. However, to simplify a bit the presentation, when the starting point of the interval is zero we drop the super--index off and simply write $\tT$.
\smallskip

We look now for solutions of the operator equation
\begin{equation*}
  U=\tT(U),
\end{equation*}
that will be called {\it mild path--area solutions} to \eqref{grds_eq01}. In order to do that, we begin by introducing the phase space. Take a fixed $\omega\in C_{\beta^\prime}([0,T];V)$ and consider $\gamma$ such that $\alpha<\gamma<1$.
Denote by $(W_{0,T},\|\cdot\|_W)$ the subspace of elements $U=(u,v)$ of the Banach--space $C_{\beta,\sim}([0,T];V)\times C_{\beta+\beta^\prime,\sim}(\ \Delta_{0,T};V\otimes V)$ such that the Chen--equality (\ref{chen}) holds. Let us also consider a subset $\hat W_{0,T}$ of this space given by the limit points in this space of the set
\begin{align}\label{localeq12}
\begin{split}
  \{(u_{n},(u_{n}\otimes \omega_{n}))&:n\in\NN,u_{n}\in  C_{\gamma}([0,T];V), \,u_n(0)\in D((-A)),\\& (\omega_n,(\omega_n\otimes_S\omega_n))\,\text{satisfies }{\bf (H3)}\}.
\end{split}
\end{align}
Note that this set of limit points is a  subspace of
$C_{\beta,\sim}([0,T];V)\times C_{\beta+\beta^\prime,\sim}(\Delta_{0,T};V\otimes V)$ which is closed. Hence $\hat W_{0,T}$ itself is a complete metric space depending on $\omega$ with a metric generated by the norm of $C_{\beta,\sim}([0,T];V)\times C_{\beta+\beta^\prime,\sim}(\Delta_{0,T};V\otimes V)$:
\begin{equation*}
 d_{\hat W}(U,\hat U)  =\|U-\hat U\|_W=\|u-\hat u\|_{\beta,\sim}+\|v-\hat v\|_{\beta+\beta^\prime,\sim},
\end{equation*}
for $U,\,\hat U \in \hat W_{0,T}$. In addition, elements $U\in\hat W_{0,T}$ satisfy the Chen--equality
\eqref{chen} with respect to $\omega$. \\

\begin{remark}\label{nota}
We would like to emphasize that the choice of the space $\hat W_{0,T}$ ensures the following additivity formula:
\begin{equation*}\label{grds_eq5}
  S(t-\tau)\int_0^\tau S(\tau-r)G(u(r))d\omega(r)+\int_\tau^tS(t-r)G(u(r))d\omega(r)=\int_0^t S(t-r)G(u(r))d\omega(r),
\end{equation*}
which is derived as a result of that the approximative integrals related to the space $\hat W_{0,T}$ enjoy such additivity property.
As pointed out in \cite{GaLuSch}, for the original integral it is quite involved to show the additivity, which is necessary to prove the uniqueness of the local solution as well as and the Chen--equality for the solution of \eqref{grds_eq011}.

\end{remark}

Let us mention some properties of $\tT$, see \cite{GaLuSch} for the proof:

\begin{lemma}\label{grds_l2b}
Suppose {\bf (H)} holds.
(i) For two elements $\tilde U=(\tilde u,\tilde v),\,\hat U=(\hat u,\hat v)\in {\hat W_{0,T}}$ such that $\tilde u(0)=\tilde u_0\in V$, $\hat u(0)=\hat u_0\in V$:
    \begin{align*}
      \|\tT &(\tilde U,\omega,(\omega\otimes_S\omega),\tilde u_0)-\tT(\hat U,\omega,(\omega\otimes_S\omega),\hat u_0)\|_{W}\\
      \le & c|\tilde u_0-\hat u_0| +cT^{\beta^\prime}(1+\|\tilde U\|_W^2+\|\hat U\|_W^2)\|\tilde U-\hat U\|_W,
    \end{align*}
    where the constant $c$ depends on $T,\,|||\omega|||_{\beta^{\prime\prime}}$, $|||\omega|||_{\beta^{\prime}}$ and $\|(\omega\otimes_S\omega)\|_{2\beta^\prime}$
    for $\beta^\prime<\beta^{\prime\prime}<H$.

\smallskip

(ii) Let $U\in {\hat W_{0,T}}$ and assume that $(\omega_n,(\omega_n\otimes_S\omega_n))$ satisfies {\bf(H3)}. Then
\begin{equation*}
  \lim_{n\to\infty}\|\tT(U,\omega_n,(\omega_n\otimes_S\omega_n),u_0)-\tT(U,\omega,(\omega\otimes_S\omega),u_0)\|_W=0.
\end{equation*}
This convergence is uniform for $U\in {\hat W_{0,T}}$  contained in a bounded set.
In addition, $\tT(U,\omega,(\omega\otimes_S\omega),u_0)\in \hat W_{0,T}$.

\end{lemma}

Next we establish the existence and uniqueness of a local mild solution as well as its regularity:

\begin{theorem}\label{grds_t2}
Assume the Hypothesis {\bf (H)}. Then for any $u_0\in V$ and $\omega\in C_{\beta^\prime}([0,T];V)$ there exists a time $T_0>0$ such that
\begin{equation*}
  U=\tT(U)
\end{equation*}
has a unique mild path--area  solution $U^0=(u^0,v^0)\in \hat W_{0,T_0}$.

\smallskip

Let $\hat V=V_\kappa$, then for $0<t \le T_0$ we have that $u^0(t)\in V_\kappa$. In particular,
\begin{align}\label{greg}
\begin{split}
|u^0(T_0)|_{V_\kappa}\le c\frac{|u_0|}{T_0^\kappa}+c(1+|||\omega|||_{\beta^\prime})T_0^{\beta^\prime}(1+||U^0||_{W}^2).
\end{split}
\end{align}
In addition, we obtain a sequence $(u_{n}^0)_{n\in\NN}$ with $ \lim_{n\to\infty} |u_n^0-u_0|=0$ such that
\begin{equation*}
  \lim_{n\to\infty}|u^0(T_0)-u^0_n(T_0)|_{V_\kappa}=0.
\end{equation*}
\end{theorem}

\begin{proof}
The first and third part of the statement have been proven in \cite{GaLuSch} and hence here we only prove the second one.
Notice that for $\kappa \ge 0$ and any $t>0$ it holds
$$|S(t) u_0|_{V_\kappa}\leq c\frac{|u_0|}{t^\kappa}.$$
Now, we estimate the $V_\kappa$-norm of the integral term:
\begin{align*}
\int_0^tS(t-r)G(u^0(r))d\omega&=(-1)^\alpha\int_0^t \hat D_{0+}^\alpha (
S(t-\cdot)G(u^0(\cdot)))[r]D_{t-}^{1-\alpha}\omega_{t-}[r]dr \\
&-(-1)^{2\alpha-1}\int_0^t D_{0+}^{2\alpha-1}(
S(t-\cdot)DG(u^0(\cdot)))[r] D_{t-}^{1-\alpha}\mathcal{D}
_{t-}^{1-\alpha}v^0[r]dr\\
&=:I_{1}(t)+I_{2}(t).
\end{align*}

Consider $\alpha^\prime>0$ such that $\beta^\prime+\alpha>\alpha^\prime>\alpha$, then thanks to Lemma \ref{locall2}
\begin{align*}
\|\hat D_{0+}^\alpha&(S(t-\cdot)G(u^0(\cdot)))[r]\|_{L_2(V,V_\kappa)}\\
&\le c\bigg(%
\frac{\|S(t-r)G(u^0(r))\|_{L_2(V,V_\kappa)}}{r^\alpha} + \int_0^r \frac{\|(S(t-r)-S(t-q))G(u^0(r))\|_{L_2(V,V_\kappa)}}{(r-q)^{\alpha+1}}dq\\
&\qquad + \int_0^r \frac{ \|S(t-q)(G(u^0(r))-G(u^0(q))-DG(u^0(q))(u^0(r)-u^0(q)))\|_{L_2(V,V_\kappa)}}{(r-q)^{\alpha+1}}dq\bigg)\\
&\le c\bigg(%
\frac{\|G(u^0(r))\|_{L_2(V,V_\kappa)}}{r^\alpha} + \int_0^r \frac{(r-q)^{\alpha^\prime}\|G(u^0(r))\|_{L_2(V,V_\kappa)}}{(t-r)^{\alpha^\prime}(r-q)^{\alpha+1}}dq\\
&\qquad + \int_0^r \frac{ \|(G(u^0(r))-G(u^0(q))-DG(u^0(q))(u^0(r)-u^0(q)))\|_{L_2(V,V_\kappa)}}{(r-q)^{\alpha+1}}dq\bigg)\\
&\leq c\bigg(\frac{c_G}{r^\alpha}+\frac{c_G}{ (t-r)^{\alpha^\prime} }  \int_0^r \frac{ (r-q)^{\alpha^\prime}}{(r-q)^{\alpha+1}} dq +c_{D^2G}\|u^0\|_{\beta,\sim}^2  \int_0^r \frac{ (r-q)^{2\beta} q ^{-2\beta} }{(r-q)^{\alpha+1}} dq\bigg)\\
&\le c ( 1+\|u^0\|^2_{\beta,\sim} ) \bigg(\frac{1}{r^\alpha}+ \frac{1}{(t-r)^{\alpha^\prime}} r^{\alpha^\prime-\alpha}\bigg).
\end{align*}
Since trivially $|D_{t-}^{1-\alpha} \omega_{t-} [r]|\leq c(t-r)^{\beta^\prime+\alpha-1}|||\omega|||_{\beta^\prime}$, for $t>0$ we have
$$|I_{1}(t)|_{V_\kappa} \leq c |||\omega|||_{\beta^\prime} t^{\beta^\prime} (1+\|u^0\|^2_{\beta,\sim} )\leq c |||\omega|||_{\beta^\prime} t^{\beta^\prime} (1+\|U^0\|^2_W ).$$

Similarly, we obtain
\begin{align*}
&\|D_{0+}^{2\alpha-1}(S(t-\cdot)DG(u^0(\cdot)))[r]\|_{L_2(V,V_\kappa)}
 \le c\bigg(\frac{
\|S(t-r)DG(u^0(r))\|_{L_2(V,V_\kappa)}}{r^{2\alpha-1}}\\
&\quad+\int_0^r \frac{\|(S(t-r)-S(t-q))DG(u^0(r))\|_{L_2(V,V_\kappa)}}{(r-q)^{2\alpha}}dq\\
&\quad+\int_0^r \frac{\|S(t-q)(DG(u^0(r))-DG(u^0(q)))\|_{L_2(V,V_\kappa)}}{(r-q)^{2\alpha}}dq\bigg)\\
&\leq c\bigg(\frac{
c_{DG}}{r^{2\alpha-1}}+c_{DG}\int_0^r (t-r)^{-\beta}(r-q)^{\beta-2\alpha}dq+c_{D^2G}\|u^0\|_{\beta,\sim}\int_0^r \frac{(r-q)^{\beta}q^{-\beta}}{(r-q)^{2\alpha}}dq\bigg)\\
&\le c(r^{1-2\alpha}+(t-r)^{-\beta} r^{\beta+1-2\alpha}+\|u^0\|_{\beta,\sim}r^{1-2\alpha}).
\end{align*}

Since $\|D^{1-\alpha}_{t-}\dD^{1-\alpha}_{t-} v^0[r]\|\leq c (\|v^0\|_{\beta+\beta^\prime, \sim} +\|u^0\|_{\beta,\sim}|||\omega|||_{\beta^\prime}) r^{-\beta} (t-r)^{\beta+\beta^\prime+2\alpha -2}$, for $t>0$  we get
\begin{align*}
    |I_{2}(t)|_{V_\kappa}\le&c t^{\beta^\prime}(1+|||\omega|||_{\beta^\prime}) (1+\|U^0\|_W^2),
\end{align*}
and this concludes the proof. The second part of the proof follows similar by the same techniques proving Lemma \ref{ex3}.
\end{proof}

Hence $U^0=(u^0,v^0) \in \hat W_{0,T_0}$ is the (local) mild solution to $U=\tT(U)$ obtained in Theorem \ref{grds_t2}. We denote $I_0=[0,T_0]$.


\section{Pathwise global solutions to SEEs driven by a H\"older path}

In this section we want to go one step further and present the existence of a {\it global} solution to (\ref{grds_eq011}). The key fact of this result will be that we can build a global solution provided that the initial condition is regular enough. Starting with the initial condition $u_0\in V$, what we know according to Theorem \ref{grds_t2}, and in particular to (\ref{greg}), is that the corresponding local solution to (\ref{grds_eq011}) satisfies $u^0(T_0)\in V_\kappa$. Therefore, in a second step we can pick precisely as initial condition $u^0(T_0)$ and construct the corresponding mild solution, which will turn out to be global thanks to a {\it concatenation} procedure which will be described below.

Throughout all this section, we take $\hat V$ to be the space $V_\kappa$, where the condition on $\kappa$ will be detailed below in the next results.

To start with, we need to consider a new phase space: for $\omega \in C_{\beta^\prime}([0,T];V)$, let $(X_{0,T},\|\cdot\|_X)$ be the subspace of elements $U=(u,v)$ of the Banach--space $C_{\beta}([0,T];V)\times C_{\beta+\beta^\prime}(\ \Delta_{0,T};V\otimes V)$ such that the Chen--equality (\ref{chen}) holds. Straightforwardly $X_{0,T}$ has the norm
\begin{equation*}
  \|U\|_X=\|u\|_{\beta}+\|v\|_{\beta+\beta^\prime}=\|u\|_C+|||U|||_X,
\end{equation*}
where
\begin{equation*}
  |||U|||_X=|||u|||_{\beta}+\|v\|_{\beta+\beta^\prime}.
\end{equation*}
Let $\hat X_{0,T}$ be the (closed) subspace of $X_{0,T}$ with elements $U=(u,v)$ given by the limit points in $X_{0,T}$ of the set
\begin{align}\label{localeq12a}
\begin{split}
  \{(u_{n},(u_{n}\otimes \omega_{n}))&:n\in\NN,u_{n}\in  C_{\gamma}([0,T];V), \,u_n(0)\in V_{\kappa},\\& (\omega_n,(\omega_n\otimes_S\omega_n))\,\text{satisfies }{\bf (H3)}\}
\end{split}
\end{align}
where $\alpha<\gamma<1$. $\hat X_{0,T}$ is a complete metric space with the metric
\begin{equation*}
 d_{\hat X}(U,\hat U) =\|U-\hat U\|_X=||u-\hat u||_{\beta}+\|v-\hat v\|_{\beta+\beta^\prime}
\end{equation*}
for $U,\,\hat U \in \hat X_{0,T}$. Moreover, elements of this space satisfy the Chen-equality \eqref{chen} with respect to $\omega$.
Note that in the space $\hat X_{0,T}$ the additivity of the pathwise integral holds in the same sense as in Remark \ref{nota}.

Consider now only elements of this space with fixed initial value $u(0)=u_0\in V_\kappa$, then we obtain a complete metric space $\hat X_{0,T,u_0}$ with metric
\begin{equation}\label{newnorm}
 d_{\hat X_{0,T},u_0}(U,\hat U) =|||U-\hat U|||_X=|||u-\hat u|||_{\beta}+\|v-\hat v\|_{\beta+\beta^\prime}.
\end{equation}
We also use the notation $u_0$ to describe the constant function on $[0,T]$. Then $U_0:=(u_0,0)\in \hat X_{0,T,u_0}$ and
\begin{equation*}
  d_{\hat X_{0,T},u_0}(U,U_0) =|||U-U_0|||_X=|||u|||_{\beta}+\|v\|_{\beta+\beta^\prime}=|||U|||_X.
\end{equation*}

In particular, the closed ball in $\hat X_{0,T,u_0}$ with center $U_0$ and radius $R$ is given by

\begin{equation*}
B_{\hat X_{0,T},u_0}(R)=  \{U=(u,v)\in \hat X_{0,T,u_0}:|||u|||_{\beta}+\|v\|_{\beta+\beta^\prime}=|||U|||_X\le R\}.
\end{equation*}
In a similar manner we can define qualitatively the spaces $\hat X_{\check T,\hat T,\check u}$ over intervals $[\check T,\hat T]$ and with respect to elements $\check U=(\check u,0)$, with $\check u\in V_\kappa$.\\

The following two lemmas will play an important role to prove the existence of a global solution to (\ref{grds_eq011}) provided that the initial condition belongs to an adequate space $V_\kappa$. Their proofs are omitted since they are quite similar to the corresponding results when having an initial condition in $V$ and $U\in \hat W_{0,T}$, see \cite{GaLuSch}.

\begin{lemma}\label{ex1old} Suppose that Hypothesis ${\bf (H)}$ holds. Then for any $1\geq \kappa\,\geq \beta$ and any $T>0$ there exists $c>0$ such that for $u_0\in V_\kappa$ and $U\in \hat X_{0,T,u_0}$ we have
\begin{align*}
|||{\mathcal T}(U)|||_{X} &\le c(T^{\beta^\prime-\beta}(1+T^{2\beta}|||U|||_X^2)+T^{\kappa-\beta}|u_0|_{V_\kappa}),\\
|{\mathcal T}_1(U)(T)|_{V_\kappa} &\le
cT^{\beta^\prime}(1+T^{2\beta}|||U|||_X^2)+|u_0|_{V_\kappa}.
\end{align*}
The constant $c$ depends on $T$, $|||\omega|||_{\beta^{\prime}}$ and $\|(\omega\otimes_S\omega)\|_{2\beta^\prime}$.
\end{lemma}

Let us only point out that, by Lemma \ref{grds_l1}, the map $t\mapsto S(t)u_0$ is $\beta$--H{\"o}lder--continuous, which is necessary to estimate the $X$-norm of $\tT$. More precisely, we have
\begin{equation*}
|||S(\cdot)u_0|||_{\beta}\le T^{\kappa-\beta}|u_0|_{V_\kappa}.
\end{equation*}

\begin{lemma}\label{ex3}
Suppose that Hypothesis ${\bf (H)}$ holds.

(i) For any $1\geq \kappa\,\geq \beta$ and any $T>0$ there exists $c>0$ such that for $\tilde U=(\tilde u,\tilde v),\,\hat U=(\hat u,\hat v)\in \hat X_{0,T}$ with $\tilde u(0)=\tilde u_0\in V_\kappa,\,\hat u(0)=\hat u_0\in V_\kappa$,
we have
 \begin{align*}
|||{\mathcal T}(\tilde U)-{\mathcal T}(\hat U)|||_X& \le cT^{\beta^\prime-\beta}(1+T^{2\beta}(|||\tilde U|||_X^2+|||\hat U|||_X^2))(|\tilde u_0-\hat u_0|_{V_\kappa}\\
&+|||\tilde U-\hat U|||_X)+T^{\kappa-\beta}|\tilde u_0-\hat u_0|_{V_\kappa}
\end{align*}
and
\begin{align*}
|{\mathcal T}_1(\tilde U)(T)-{\mathcal T}_1(\hat U)(T)|_{V_\kappa} &\le cT^{\beta^\prime}(1+T^{2\beta}(|||\tilde U|||_X^2+|||\hat U|||_X^2))(|\tilde u_0-\hat u_0|_{V_\kappa}\\
&\quad +|||\tilde U-\hat U|||_X)+c|\tilde u_0-\hat u_0|_{V_\kappa}.
\end{align*}
\end{lemma}
(ii) Let $U\in {\hat X_{0,T}}$ and assume that $(\omega_n,(\omega_n\otimes_S\omega_n))$ satisfies {\bf(H3)}. Then
\begin{equation*}
  \lim_{n\to\infty}|||\tT(U,\omega_n,(\omega_n\otimes_S\omega_n),u_0)-\tT(U,\omega,(\omega\otimes_S\omega),u_0)|||_X=0.
\end{equation*}
This convergence is uniform for $U\in {\hat X_{0,T}}$, $u_0\in V_\kappa$ contained in a bounded set.

\begin{proof}

We only give an idea of the proof of the first inequality of (i), which is based of the different estimates for $G$ coming from Lemma \ref{locall2}.

Trivially, $|||S(\cdot) \tilde u_0-S(\cdot)\hat u_0|||_{\beta}\leq T^{\kappa-\beta}|\tilde u_0-\hat u_0|_{V_\kappa}.$ Denoting $\Delta u=\tilde u-\hat u$,
\begin{align*}
    \|G(\tilde u(r)) &-G(\hat u(r))\|_{L_2(V,V_\kappa)}\le c_{DG}(|\tilde u(r)-\tilde u_0-\hat u(r)+\hat u_0|+|\tilde u_0-\hat u_0|)\\
    &\le c_{DG}\bigg(\sup_{0\le q<r\le T}\frac{|\tilde u(r)-\hat u(r)-(\tilde u(q)-\hat u(q))|}{(r-q)^{\beta}}T^\beta+|\tilde u_0-\hat u_0|\bigg)\\&=c_{DG}(|||\Delta u|||_\beta T^\beta+|\tilde u_0-\hat u_0|)\le c c_{DG}(|||\Delta u|||_\beta T^\beta+|\tilde u_0-\hat u_0|_{V_\kappa})
\end{align*}
and similar for $DG(u(r))$. Note that the last inequality above is a consequence of the continuous embedding $V_\kappa \subset V$. Moreover, thanks to Lemma \ref{locall2} we get
\begin{align*}
    &\|DG(\tilde u(r))-DG(\hat u(r))-(DG(\tilde u(q))-DG(\hat u(q)))\|_{L_2(V,V_\kappa)}\\
   \le& c_{D^2G}|||\Delta u|||_{\beta}(r-q)^\beta+ c c_{D^3G}(|||\Delta u|||_{\beta}
   T^\beta+|u_0^1-u_0^2|_{V_\kappa}) (|||\tilde u|||_\beta+|||\hat u|||_{\beta}) (r-q)^\beta
    \end{align*}
and
\begin{align*}
    \|G(\tilde u(r))&-G(\tilde u(q))-DG(\tilde u(q))(\tilde u(r)-\tilde u(q))
    \\&-(G(\hat u(r))-G(\hat u(q))-DG(\hat u(q))(\hat u(r)-\hat u(q)))\|_{L_2(V,V_\kappa)}\\
    &\le c_{D^2G}
    (|||\tilde u|||_\beta+|||\hat u|||_\beta)|||\Delta u|||_\beta(r-q)^{2\beta}\\
    &+c c_{D^3G}|||\hat u|||_\beta(|||\Delta u|||_\beta T^\beta+|\tilde u_0-\hat u_0|_{V_\kappa})(2|||\tilde u|||_\beta+|||\hat u |||_\beta)(r-q)^{2\beta}.
\end{align*}

\end{proof}

In particular, when $\tilde U,\,\hat U\in \hat X_{0,T,u_0}$, the first inequality of Lemma \ref{ex3} (i) becomes
\begin{align}\label{sic}
|||{\mathcal T}(\tilde U)-{\mathcal T}(\hat U)|||_X& \le cT^{\beta^\prime-\beta}(1+T^{2\beta}(|||\tilde U|||_X^2+|||\hat U|||_X^2))|||\tilde U-\hat U|||_X.
\end{align}

We can also consider $\tT$ on other time intervals, that is, defined on intervals $[s,T]$ where $s>0$. Indeed, replacing $\hat X_{0,T}$ by $\hat X_{s,T}$ these operators satisfy the same estimates as $\tT$ stated in Lemma \ref{ex1old} and Lemma \ref{ex3}. As we already said in Section 3, we denote these operators by $\tT^s$.

In what follows we describe how to construct a global solution to our equation on $[0,T]$.

In a first step, we shall focus on proving the existence of a solution belonging to $\hat X_{T_0,T,u_0^1}$, where $T_0$ was defined in Theorem \ref{grds_t2}. We would like to mention that even though the above estimates of $\tT$ contain squared expressions, it is possible to construct solutions defined in adequate small time intervals $[T_{i-1},T_i]$, namely, solutions belonging to $\hat X_{T_{i-1},T_i,u^{i-1}(T_{i-1})}$ for suitable initial conditions, to further these solutions be joined together to form a whole solution in $\hat X_{T_0,T,u_0^1}$. The choice of the starting point in each step will be clear in the construction of the solution. To prove the existence of these solutions we apply Lemma \ref{ex1old} and Lemma \ref{ex3} on any of the intervals  $[T_{i-1},T_i]$. Note that the appearing constant
$c$ in these lemmas depends on  those intervals, on the ${\beta^\prime}$--norm of $\omega$ in $[T_{i-1},T_i]$ and on the ${2\beta^\prime}$--norm of $(\omega\otimes_S\omega)$ with respect to $\bar\Delta_{T_{i-1},T_i}$. However the following results show that we can choose a larger constant independent of those intervals, and only depending on the given number $T$ and the norms of $\omega$, $(\omega\otimes_S\omega)$ with respect to $[0,T]$, $\bar\Delta_{0,T}$, respectively.  In addition, $c$ will be chosen large enough so that we also take the construction of the local solution given in Theorem \ref{grds_t2} into account.
\smallskip

The way in which we join together these pieces of solutions is described in general as follows:

\begin{definition}\label{concatenation}
For $0<a<b$ we consider the {\em concatenation} $U=(u,v)$ of elements $U^1=(u^1,v^1)\in \hat X_{a,b}$ with $u^1(a)\in V_\kappa$ (or $U^1=(u^1,v^1)\in \hat W_{a,b}$ with $u^1(a)\in V$) and $U^2=(u^2,v^2)\in \hat X_{b,c}$, with $u^2(b)=u^1(b)$, defined by:
\begin{align*}
    &u(t)=\left\{\begin{array}{lcl}
    u^1(t)&:& a\le t \le b\\
    u^2(t)&:& b\le t\le c
    \end{array}\right.\\
    &v(s,t)=\left\{\begin{array}{lcr}
    v^1(s,t)&:& a\le s\le t \le b\\
    v^2(s,t)&:& b\le s\le t\le c\\
    (u^1(b)-u^1(s))\otimes_V(\omega(t)-\omega(b))+v^1(s,b)+v^2(b,t)&:& s\le b<t.
    \end{array}\right.
\end{align*}
\end{definition}
Note that in the previous definition $v$ has been defined according to the Chen--equality.  It is straightforward to check that $U$ satisfies the Chen--equality and $U\in \hat X_{a,c}$ (or  $U\in \hat W_{a,c}$ when $U^1\in \hat W_{a,b}$ with $u^1(a)\in V$). For a proof of a similar statement see Lemma \ref{BW} below.\\

In the next result we present the existence and uniqueness of a first piece of solution $U^1=(u^1,v^1)\in \hat X_{T_0,T_1,u_0^1}$, with initial condition $u_0^1 \in V_\kappa$.

Later on, we will take $u_0^1=u^0(T_0)$ where we remind that $U^0=(u^0,v^0)$ denotes the solution to (\ref{grds_eq011}) in $\hat W_{0,T_0}$ with initial condition $u_0\in V$, see Theorem \ref{grds_t2}.

\begin{theorem}\label{tex2}
Assume Hypothesis ${\bf (H)}$ holds, and that $\kappa+\beta^\prime>1$, $\kappa\leq 1$. Then for any $u_0^1\in V_\kappa$ there exists a small enough $T_1>T_0$ depending on the data such that $U=\tT^{T_0}(U)$ has a unique mild  solution $U^1=(u^1,v^1)\in \hat X_{T_0,T_1,u_0^1}$.

\end{theorem}
\begin{proof}
Define $|u_0^1|_{V_\kappa}=:\rho_0$. Let us consider $K(\rho_0)\geq 1$ such that for $K\ge K(\rho_0)$, the following relations are fulfilled:
\begin{align}\begin{split}\label{ex8}
&\rho_0+2cK^{-\beta^\prime}< K^{1-\beta^\prime},\\
    &4c^2 K^{-\beta^\prime-\beta}(K^{\beta-\kappa}K^{1-\beta^\prime}+K^{\beta-\beta^\prime})< 1,
    \\
    &cK^{\beta-\beta^\prime}(1+2K^{-2\beta}(8c^2K^{2\beta-2\kappa}K^{2-2\beta^\prime}+8c^2K^{2\beta-2\beta^\prime}))<\frac12,\\
    &cK^{-\beta^\prime}+cK^{-\beta^\prime-2\beta}
    (8c^2K^{2\beta-2\kappa}K^{2-2\beta^\prime}+8c^2K^{2\beta-2\beta^\prime})< 2cK^{-\beta^\prime}.
    \end{split}
\end{align}
Note that $\rho_0< K^{1-\beta^\prime}$ and that the last three inequalities are true due to the assumption $\kappa+\beta^\prime>1$ and {\bf(H1)}.

In what follows we prove that (\ref{ex8}) ensures that $\tT^{T_0}$ maps a ball $B_{\hat X_{T_0,T_1,u_0^1}}(R_1)$ into itself and it is a contraction on it, for $T_0$ given in Theorem \ref{grds_t2}, for a small enough $T_1$ and for an appropriate radius $R_1$. In particular, let us take $T_1=T_0+\Delta T_1$ where $\Delta T_1=K^{-1}$. Then,  by Lemma \ref{ex1old} we have
\begin{equation*}
    |||{\mathcal T}(U)|||_{X_{T_0,T_1}} \le c(\Delta T_1^{\kappa-\beta}\rho_0+\Delta T_1^{\beta^\prime-\beta}+\Delta T_1^{\beta^\prime+\beta}|||U|||_{X_{T_0,T_1}}^2).
\end{equation*}
Hence, to find a ball $B_{\hat X_{T_0,T_1,u_0^1}}(R_1)$ that will be mapped into itself we calculate the minor root $R_1$ of the quadratic equation
\begin{equation}\label{2eq}
    x=c(\Delta T_1^{\kappa-\beta}\rho_0+\Delta T_1^{\beta^\prime-\beta}+\Delta T_1^{\beta^\prime+\beta}x^2)
\end{equation}
which, according to Sohr \cite{Sohr} Page 349, is given by
\begin{equation*}
   R_1:= \frac{2c(\Delta T_1^{\kappa-\beta}\rho_0+\Delta T_1^{\beta^\prime-\beta})}{1+\sqrt{1-4c^2  \Delta T_1^{\beta^\prime+\beta}(\Delta T_1^{\kappa-\beta}\rho_0+ \Delta T_1^{\beta^\prime-\beta})}}< 2c( \Delta T_1^{\kappa-\beta}\rho_0+ \Delta T_1^{\beta^\prime-\beta}).
\end{equation*}
This root is well-defined due to the definition of $\Delta T_1$ and the first and second inequality of \eqref{ex8} , since these conditions in particular imply that
\begin{equation*}
    1-4c^2 \Delta T_1^{\beta^\prime+\beta}(\Delta T_1^{\kappa-\beta}\rho_0+\Delta T_1^{\beta^\prime-\beta})> 0.
\end{equation*}
Note that if $U\in B_{X_{T_0,T_1,u_0^1}}(R_1)$, then
\begin{equation*}
|||U|||_{X_{T_0,T_1}}^2\leq R_1^2 \leq 8c^2 \Delta T_1^{2\kappa -2\beta}\rho_0^2+8c^2 \Delta T_1^{2\beta^\prime -2\beta},
\end{equation*}
and thanks to (\ref{sic}) we obtain that $\mathcal{T}$ is a contraction on $B_{\hat X_{T_0,T_1,u_0^1}}(R_1)$ if
\begin{equation*}
    c \Delta T_1^{\beta^\prime-\beta}(1+2 \Delta T_1^{2\beta}(8c^2 \Delta T_1^{2\kappa -2\beta}\rho_0^2+8c^2 \Delta T_1^{2\beta^\prime -2\beta}))<\frac12.
\end{equation*}
The previous relation follows from the first and third inequality of \eqref{ex8}.

To see that $\tT(B_{\hat X_{T_0,T_1,u_0^1}})(R_1)\in \hat X_{T_0,T_1,u_0^1}$ we approximate an element $U=(u,v)\in B_{\hat X_{T_0,T_1,u_0^1}}(R_1)$ by a sequence $(u_n,(u_n\otimes\omega_n))$ from \eqref{localeq12a}. We note that $\tT_1((u_n,(u_n\otimes\omega_n)),\omega_n,u_n(0))(0)=u_n(0)\in V_\kappa$ and $\tT_1((u_n,(u_n\otimes\omega_n)),\omega_n,u_n(0))\in C_{\gamma}([0,T];V)$
for any $\gamma\in (0,1)$, see Pazy \cite{Pazy} Theorem 4.3.1. Therefore $\tT_2((u_n,(u_n\otimes\omega_n)),\omega_n,(\omega_n\otimes_S \omega_n),u_n(0))$ can be defined as $(\tT_1(u_n,(u_n \otimes\omega_n)),\omega_n, u_n(0))\otimes \omega_n)$ given by (\ref{grds_eq3b}). The previous considerations mean that $\tT((u_n, (u_n\otimes \omega_n)\omega_n, u_n(0))\otimes \omega_n)$ belongs to the set (\ref{localeq12a}). Now it suffices to use the first inequality of (i) and (ii) in Lemma \ref{ex3} to ensure that $\tT(U,\omega,(\omega\otimes_S\omega),u_0)\in \hat X_{T_0,T_1,u_0^1}$.

\smallskip

Therefore, we obtain that $\tT^{T_0}$ has a fixed point $U^1=(u^1,v^1)\in B_{\hat X_{T_0,T_1,u_0^1}}(R_1)$. Furthermore, by Lemma \ref{ex1old} and the first and last inequality of (\ref{ex8}), we get
\begin{align}\label{ric}
\begin{split}
    |u^1(T_1)|_{V_\kappa}&\le c \Delta T_1^{\beta^\prime}(1+ \Delta T_1^{2\beta}|||U^1|||_{X_{T_0,T_1}}^2)+|u_0^1|_{V_\kappa}\\
    &\leq c \Delta T_1^{\beta^\prime}+c \Delta T_1^{\beta^\prime+2\beta}(8c^2 \Delta T_1^{2\kappa-2\beta} \rho_0^2+8c^2\Delta T_1^{2\beta^\prime-2\beta})+\rho_0\\
    &< 2c \Delta T_1^{\beta^\prime}+\rho_0 <K^{1-\beta^\prime}.
    \end{split}
    \end{align}
\end{proof}

We want to point out that in the previous result it has been crucial to use the seminorm $|||\cdot|||_X$ instead of the norm $\|\cdot\|_X$. If we had used directly the norm then we would not have found an appropriate $K$ fulfilling (\ref{ex8}). \\

Up to now, we have obtained the existence of a local solution $U^1=(u^1,v^1)\in \hat X_{T_0,T_1,u_0^1}$ with initial condition $u_0^1 \in V_\kappa$, that is, we have found a solution whose path component $u^1$ is defined on the interval $I_1=[T_0,T_1]$ for $T_1=T_0+\frac{1}{K}$. Now we would like to repeat the same arguments than in Theorem \ref{tex2} in time intervals
$$I_2=[T_1,T_2]:=[T_1,T_1+\frac{1}{2K}], \quad \cdots \cdots \quad I_i:=[T_{i-1},T_i]=[T_{i-1},T_{i-1}+\frac{1}{iK}],$$
for $i=2,\cdots,i^\ast$ for appropriate $i^\ast$, see Theorem \ref{tglobal}.



\begin{theorem} \label{tglobal} Let $T>0$ be some number. Suppose that {\bf (H)} holds and $\kappa+\beta^\prime>1$, $\kappa\leq 1$. Suppose also that $u_0^1\in V_\kappa$ is given by $u^0(T_0)=\tT_1((u^0,v^0),\omega,u_0)(T_0),\,u_0\in V$, see Theorem \ref{grds_t2}. Then there exist $i^\ast$  intervals  $[T_0,T_1],\cdots [T_{i^\ast-1},T_{i^\ast}]$  where $T_{i^\ast-1} <T\le T_{i^\ast}$ such that on any of these intervals
\[
U^i=\tT^{T_{i-1}}(U^i,\omega,(\omega\otimes_S\omega),u^{i-1}(T_{i-1}))
\]
has a unique solution $U^i=(u^i, v^i) \in \hat X_{T_{i-1},T_{i},u^{i-1}(T_{i-1})}$. Therefore, each of these local solutions can be approximated by  the classical solutions  $U^i_n=(u_n^i, (u_n^i \otimes \omega_n))$ that satisfy
\[
U_n^i=\tT^{T_{i-1}}(U_n^i,\omega_n,(\omega_n\otimes_S\omega_n),u_n^{i-1}(T_{i-1})).
\]
\end{theorem}

\begin{proof} The proof consists of applying an induction procedure, with starting step given in the proof of Theorem \ref{tex2}. In order to do that, assume the following generalization of the assumption (\ref{ex8}): $K(\rho_0)\geq 1$ is such that for $K\ge K(\rho_0)$ and $i\in\mathbb{N}$ the following inequalities are fulfilled
\begin{align}\label{ex7}
\begin{split}
\rho_0&+\sum_{j=1}^{i}2c(Kj)^{-\beta^\prime}\le \rho_0+2cK^{-\beta^\prime}\frac{1}{1-\beta^\prime} i^{1-\beta^\prime}< (Ki)^{1-\beta^\prime},\\
    &4c^2 (Ki)^{-\beta^\prime-\beta}((Ki)^{\beta-\kappa}(Ki)^{1-\beta^\prime}+(Ki)^{\beta-\beta^\prime})< 1,\\
    &c(Ki)^{\beta-\beta^\prime}(1+2(Ki)^{-2\beta}(8c^2(Ki)^{2\beta-2\kappa}(Ki)^{2-2\beta^\prime}+8c^2(Ki)^{2\beta-2\beta^\prime}))<\frac12,\\
    &c(Ki)^{-\beta^\prime}+c(Ki)^{-\beta^\prime-2\beta}
    (8c^2(Ki)^{2\beta-2\kappa}(Ki)^{2-2\beta^\prime}+8c^2(Ki)^{2\beta-2\beta^\prime})\\
    &< 2c(Ki)^{-\beta^\prime}.
\end{split}
\end{align}
Once again, we stress that these inequalities hold due to the condition $\kappa+\beta^\prime>1$.

Assume that we have obtained pieces of local solutions $U^1=(u^1,v^1)\in \hat X_{T_0,T_1,u^{0}(T_0)}$, ..., $U^{i-1}=(u^{i-1},v^{i-1})\in \hat X_{T_{i-2},T_{i-1},u^{i-2}(T_{i-2})}$ such that
\begin{align}\label{coni}
|u^{i-1}(T_{i-1})|_{V_\kappa}< \rho_0+\sum_{j=1}^{i-1} 2c (Kj)^{-\beta^\prime}
\end{align}
for $i=2, 3, \cdots$, and $T_{i-1}<T$. Note that (\ref{coni}) is nothing but the generalization of the property (\ref{ric}).  Now we give some details of how to obtain the next local solution $U^{i}=(u^{i},v^{i})$:  we set $T_i=T_{i-1}+\Delta T_{i}$, $\Delta T_{i}=(Ki)^{-1}$ if $T_i< T$, and $T_i=T$ in other case. Then, by the first inequality of \eqref{ex7} we know that $|u^{i-1}(T_{i-1})|_{V_\kappa}<{ (Ki)}^{1-\beta^\prime}$.
Following the steps of Theorem \ref{tex2}, \eqref{ex7} together with the Banach fixed point theorem ensure the existence of a local solution $U^i=(u^i, v^i)\in B_{\hat X_{T_{i-1},T_i,u^{i-1}(T_{i-1})}}(R_{i})$ which is unique, where $R_{i}$ is the minor root of \eqref{2eq} when replacing $\rho_0$ by $(Ki)^{1-\beta^\prime}$ and $\Delta T_1$ by $\Delta T_{i}$. In addition, (\ref{coni}) and the first inequality and the last one of (\ref{ex7}) imply

\begin{align*}
|u^i(T_{i})|_{V_\kappa}&\le c \Delta T_i^{\beta^\prime}(1+ \Delta T_i^{2\beta}|||U|||_{X_{T_{i-1},T_i}}^2)+|u^{i-1}(T_{i-1})|_{V_\kappa}\\
    &\leq c \Delta T_i^{\beta^\prime}+c \Delta T_i^{\beta^\prime+2\beta}(8c^2 \Delta T_i^{2\kappa-2\beta} \Delta T_i^{2\beta^\prime-2}+8c^2\Delta T_i^{2\beta^\prime-2\beta})+|u^{i-1}(T_{i-1})|_{V_\kappa}\\
    &< 2c (Ki)^{-\beta^\prime}+\rho_0 +\sum_{j=1}^{i-1} 2c (Kj)^{-\beta^\prime}=\sum_{j=1}^{i} 2c (Kj)^{-\beta^\prime}+\rho_0 <(Ki)^{1-\beta^\prime}.
    \end{align*}

Since $\sum_i (iK)^{-1}=\infty$ there is an $i^\ast\in\NN$ such that $T_{i^\ast}\wedge T= T$.\\

Finally for the convergence observe that
\begin{align*}
&|||U_n^i-U^i|||_X\\
=&|||\tT^{T_{i-1}}(U_n^i,\omega_n,(\omega_n\otimes_S\omega_n),u_n^{i-1}(T_{i-1}))-\tT^{T_{i-1}}(U^i,\omega,(\omega\otimes_S\omega),u_n^{i-1}(T_{i-1}))|||_X\\
\le &|||\tT^{T_{i-1}}(U_n^i,\omega_n,(\omega_n\otimes_S\omega_n),u_n^{i-1}(T_{i-1}))-\tT^{T_{i-1}}(U_n^i,\omega_n,(\omega_n\otimes_S\omega_n),u^{i-1}(T_{i-1}))|||_X\\
+&|||\tT^{T_{i-1}}(U_n^i,\omega_n,(\omega_n\otimes_S\omega_n),u^{i-1}(T_{i-1}))-\tT^{T_{i-1}}(U_n^i,\omega,(\omega\otimes_S\omega),u^{i-1}(T_{i-1}))|||_X\\
+&|||\tT^{T_{i-1}}(U_n^i,\omega,(\omega\otimes_S\omega),u^{i-1}(T_{i-1}))-\tT^{T_{i-1}}(U^i,\omega,(\omega\otimes_S\omega),u^{i-1}(T_{i-1}))|||_X.
\end{align*}
Now we can apply Lemma \ref{ex3} to estimate these terms:
\begin{align*}
|||U_n^i & -U^i|||_X\leq  c \Delta T_i^{\beta^\prime-\beta}(1+2\Delta T_i^{2\beta}||| U_n^i|||_X^2)|u_n^{i-1}(T_{i-1})-u^{i-1}(T_{i-1})|_{V_\kappa}\\
+&c\Delta T_i^{\kappa-\beta}|u_n^{i-1}(T_{i-1})-u^{i-1}(T_{i-1})|_{V_\kappa}\\
+&|||\tT^{T_{i-1}}(U_n^i,\omega_n,(\omega_n\otimes_S\omega_n),u^{i-1}(T_{i-1}))-\tT^{T_{i-1}}(U_n^i,\omega,(\omega\otimes_S\omega),u^{i-1}(T_{i-1}))|||_X\\
+&c \Delta T_i^{\beta^\prime-\beta}(1+\Delta T_i^{2\beta}(||| U_n^i|||_X^2+||| U^i|||_X^2)|||U_n^i-U^i|||_X.
\end{align*}

For the solution $U^{i-1}$on the interval  $[T_{i-2},T_{i-1}]$ we can assume that $|u_n^{i-1}(T_{i-1})-u^{i-1}(T_{i-1})|_{V_\kappa}$ tends to zero,  which gives the convergence of the first and second terms. This is also true for $U^0$ considered in Section 4. The convergence of the third expression follows by Lemma \ref{ex3} (ii). For the last expression, we mention that $U_n^i$ satisfies similar a priori estimates as $U^i$, which follow by Lemma \ref{ex1old}. In particular, thanks to  {\bf (H3)} the constant $c$ for $U^i_n$ which depends on $\omega_n,\,(\omega_n\otimes_S\omega_n)$ can be chosen very close to the constant for $U^i$, and therefore $|||U_n ^i|||_X$ is very close to $|||U^i|||_X$ such that
\begin{equation*}
  c \Delta T_i^{\beta^\prime-\beta}(1+\Delta T_i^{2\beta}(||| U_n^i|||_X^2+||| U^i|||_X^2))\le q
\end{equation*}
where $q<1$. Summarizing we find the desired convergence, since $\|U_n^i-U^i\|_X$ is equivalent to
\begin{equation*}
|u_n^i(T_{i-1})-u^i(T_{i-1})|+|||U_n^i-U^i|||_X
\end{equation*}
which can be estimated by $|u_n^i(T_{i-1})-u^i(T_{i-1})|_{V_\kappa}+|||U_n^i-U^i|||_X$, and this sum goes to zero.

\end{proof}

Then we can prove the following result:

\begin{lemma}\label{BC}
The iterated concatenation $U=(u,v)$ of $\,U^i=(u^i, v^i)$ for $i=0,1,\cdots, i^\ast$ (in the sense of Definition \ref{concatenation}) is such that $U\in C_{\beta,\sim}([0,T];V)\times C_{\beta+\beta^\prime,\sim}(\Delta_{0,T};V\otimes V)$.
\end{lemma}
\begin{proof}
Given $(s,t)\in\Delta_{0,T}$ we know that
\begin{equation*}
  T_{i_s-1}\le s\le T_{i_s} < \cdots < T_{i_t-1} < t\le T_{i_t}
\end{equation*}
where  $I_{t_{i_s}}=[T_{i_s-1},T_{i_s}]$ and so on. By an iterated application of the Chen--equality we get
\begin{align*}
  v(s,t) = &v^{{i_s}}(s,T_{i_s})+v^{{i_s+1}}(T_{i_s},T_{i_s+1})+\cdots +v^{{i_t}}(T_{i_t-1},t) \\
   & +(u^{{i_s}}(T_{i_s})-u^{{i_s}}(s))\otimes_V(\omega(t)-\omega(T_{i_s}))\\
   &+(u^{{i_s+1}}(T_{i_s+1})-u^{{i_s+1}}(T_{i_s}))\otimes_V(\omega(t)-\omega(T_{i_s+1}))\\
  & \quad \vdots  \\
&+(u^{{i_t-1}}(T_{i_t-1})-u^{{i_t-1}}(T_{i_t-2}))\otimes_V(\omega(t)-\omega(T_{i_t-1})).
\end{align*}
Now taking the $\|\cdot\|$-norm, applying the triangle inequality, multiplying the expression by $s^\beta/(t-s)^{\beta+\beta^\prime}$ and taking the supremum for $0<s< t<T$ we obtain
\begin{align*}
  \|v\|_{\beta+\beta^\prime,\sim}  \le &\|v^0\|_{\beta+\beta^\prime,\sim,\Delta_{I_0}}+\|u^0\|_{\beta,\sim,I_0}|||\omega|||_{\beta^\prime,[0,T]}\\
   & +T^\beta \sum_{i=1}^{i^\ast}\|v^i \|_{\beta+\beta^\prime,\bar \Delta_{I_i}}+T^\beta \sum_{i=1}^{i^\ast}|||u^i|||_{\beta,I_i}|||\omega|||_{\beta^\prime,[0,T]}<\infty.
\end{align*}
Moreover, for the first component of $U$ we get
\begin{align*}
u(t)-u(s)=&u^{{i_t}}(t)-u^{{i_s}}(s)\\
=&(u^{{i_t}}(t)- u^{{i_t}}(T_{i_t-1})) +(u^{{i_t}}(T_{i_t-1}) - u^{{i_t-1}}(T_{i_t-2}))+\cdots\\
&+ (u^{{i_s}}(T_{i_s})-u^{{i_s}}(s))
\end{align*}
and therefore
\begin{align*}
  \|u\|_{\beta,\sim} & \le \|u^0\|_{\beta,\sim,I_0}+ \sum_{i=1}^{i^\ast} \sup_{t\in [T_{i-1},T_i]}|u^i( t)|+ T^\beta \sum_{i=1}^{i^\ast}|||u^i|||_{\beta,I_i}\\
 & \le \|u^0\|_{\beta,\sim,I_0}+\sum_{i=1}^{i^\ast}|u^i(T_{i-1})|_{V_\kappa}+ 2 T^\beta \sum_{i=1}^{i^\ast}|||u^i|||_{\beta,I_i}.
\end{align*}
Note that in the last expression we have estimated $ \sup_{t\in [T_{i-1},T_i]}|u^i( t)|$ in terms of $|u^i(T_{i-1})|_{V_\kappa}$ and $|||u^i|||_\beta$.
\end{proof}

\begin{lemma}\label{BW}
The iterated concatenation $U=(u,v)$ of $\,U^i=(u^i, v^i)$ for $i=0,1,\cdots, i^\ast$ (in the sense of Definition \ref{concatenation}) verifies $U\in \hat W_{0,T}$.
\end{lemma}

\begin{proof}
Since $U^i \in \hat X_{T_{i-1},T_i,u^{i-1}(T_{i-1})}$ (and $U^0\in \hat W_{0,T_0}$) there exists $U^i_n=(u^i_n, (u^i\otimes \omega_n))$ such that $u^i_n \in C_\gamma([0,T];V)$, with $u_n^i(T_{i-1})\in V_\kappa$, $(\omega_n,(\omega_n\otimes_S \omega_n))$ fulfilling {\bf(H3)}, such that
\begin{equation}\label{app}
\lim_{n\to \infty} (u^i_n , u^i_n \otimes \omega_n )=(u^i, v^i ) \quad \text{ in $C_{\beta}(I_i;V)\times C_{\beta+\beta^\prime}(\bar \Delta_{I_i};V\otimes V)$},
\end{equation}
see (\ref{localeq12}). When $i=0$, taking $u_n^0(0)\in D((-A))$,
\begin{align*}
\lim_{n\to \infty} (u^0_n , u^0_n \otimes \omega_n )=(u^0, v^0 ) \quad \text{ in $C_{\beta,\sim}(I_0;V)\times C_{\beta+\beta^\prime,\sim}(\Delta_{I_0};V\otimes V)$}.
\end{align*}
Moreover, thanks to Theorem \ref{BC} we can choose for these approximating sequences the classical solutions
\begin{equation*}
 (u_n^i,u_n^i\otimes \omega_n)= U_n^i=\tT(U_n^i,\omega_n,(\omega_n\otimes_S\omega_n),u_n^i(T_{i-1})).
\end{equation*}
Furthermore, according to Definition \ref{concatenation}, we can consider the concatenation $U_n=(u_n,v_n)$. Then $u_n$ can be seen as the classical solution to (\ref{grds_eq011}) driven by the piecewise linear continuous path $\omega_n$ on $[0,T]$. It is easy to see that $u_n\in C_{\beta,\sim}([0,T];V)$, and by the smoothness of $\omega_n$
\begin{equation*}
  v_n(s,t)=(u_n\otimes \omega_n)(s,t)=\int_s^t(u_n(r)-u_n(s))\otimes_Vd\omega_n
\end{equation*}
such that $U_n$ satisfies the operator equation
\begin{align}\label{conv2}
U_n=\tT(U_n,\omega_n,(\omega_n\otimes_S\omega_n),u_n(0))
\end{align}
on $[0,T]$ where $(u_n(0))_{n\in\NN}$ tends to $u_0$ in $V$.
Let us show that
\begin{align}\label{conv}
\lim_{n\to \infty}\|U_n-U\|_W=0,
\end{align}
which in particular implies that $U\in \hat W_{0,T}$. To this end, let us focus on the norm of the difference of area components. In fact, similarly to the proof of Lemma \ref{BC}, we have
\begin{align*}
  \|v-v_n\|_{\beta+\beta^\prime,\sim}  \le& \|v^0-v_n^0\|_{\beta+\beta^\prime,\sim,\Delta_{I_0}}+\|u^0-u_n^0\|_{\beta,\sim,I_0}|||\omega|||_{\beta^\prime,[0,T]}
\\&+\|u_n^0\|_{\beta,\sim,I_0}|||\omega_n-\omega|||_{\beta^\prime,[0,T]}
  \\
   & +T^\beta \sum_{i=1}^{i^\ast}\|v^{i}-v_n^{i}\|_{\beta+\beta^\prime,\bar \Delta_{I_i}}+T^\beta \sum_{i=1}^{i^\ast}|||u^{i}-u_n^{i}|||_{\beta,I_i}|||\omega|||_{\beta^\prime,[0,T]}\\
&+T^\beta \sum_{i=1}^{i^\ast}|||u_n^{i}|||_{\beta,I_i}|||\omega_n-\omega|||_{\beta^\prime,[0,T]},
\end{align*}
so it suffices to take into account the convergence properties of the beginning of this lemma to conclude that $v_n\to v$ in $C_{\beta+\beta^\prime,\sim}(\Delta_{0,T},V\otimes V)$. In a similar way we obtain
\begin{equation*}
  \lim_{n\to\infty}\|u_n-u\|_{\beta,\sim}=0.
\end{equation*}
\end{proof}

Finally, we can state the global existence result when starting with a initial condition in $V$:

\begin{theorem}\label{tex2b}
Suppose Hypothesis ${\bf (H)}$ holds, and that $\kappa+\beta^\prime>1$, $\kappa \leq 1,$
and $u_0\in V$. Then there exists a unique global solution $U=(u,v)\in \hat W_{0,T}$ to $\tT(U)=U$. Therefore, the path component $u$ is a global solution to (\ref{grds_eq011}).
\end{theorem}

\begin{proof} By Theorem \ref{grds_t2}, we obtain the existence of a (local) mild solution $U^0=(u^0,v^0)$ in $\hat W_{0,T_0}$ for some $T_0>0$. A key feature at that point is that the path component evaluated in the final instant belongs to a more regular space, since the solution satisfies that $u^0(T_0) \in V_\kappa$, see Theorem \ref{grds_t2}.

Then taking $u_0^1=u^0(T_0) \in V_\kappa$ as initial data, we can construct $U^i=(u^i, v^i)$, $i=1,\cdots, i^\ast$ according to Theorem \ref{tglobal}, that later on will be all concatenated turning into $U=(u,v)$. This $U$ belongs to $\hat W_{0,T}$ thanks to Lemma \ref{BW}. In order to finish the proof we check that
\begin{equation}\label{grds_eq12}
  U=\tT(U,\omega,(\omega\otimes_S\omega),u_0).
\end{equation}
This relationship follows since by Lemma \ref{grds_l2b} (i)
\begin{align*}
\|\tT&(U,\omega,(\omega\otimes_S\omega),u_0)-\tT(U_n,\omega_n,(\omega_n\otimes_S\omega_n),u_n(0))\|_W\\
 \leq& \|\tT(U,\omega,(\omega\otimes_S\omega),u_0)-\tT(U_n,\omega,(\omega\otimes_S\omega),u_0)\|_W\\
&+\|\tT(U_n,\omega,(\omega\otimes_S\omega),u_0)-\tT(U_n,\omega_n,(\omega_n\otimes_S\omega_n),u_0)\|_W\\
&+\|\tT(U_n,\omega_n,(\omega_n\otimes_S\omega_n),u_0)-\tT(U_n,\omega_n,(\omega_n\otimes_S\omega_n),u_n(0))\|_W\\
\leq& cT^{\beta^\prime}(1+\|U\|_W^2+\|U_n\|_W^2)\|U-U_n\|_W\\
&+\|\tT(U_n,\omega,(\omega\otimes_S\omega),u_0)-\tT(U_n,\omega_n,(\omega_n\otimes_S\omega_n),u_0)\|_W\\
&+|u_0-u_n(0)|.
\end{align*}
From (\ref{conv}) and Lemma \ref{grds_l2b} (ii) we can conclude that the right hand side of the last inequality tends to zero
due to the fact that  $(\|U_n\|_W)_{n\in\NN}$ is bounded. Hence \eqref{grds_eq12} follows from (\ref{conv2}) and (\ref{conv}).
This mild path-area solution is unique in $\hat W_{0,T}$, see \cite{GaLuSch}.
\end{proof}


\section{random dynamical systems}\label{RDSs}

In this section we prove that \eqref{grds_eq011} generates a random dynamical system. The fact that a stochastic differential equation generates a random dynamical system allows to use a huge machinery to investigate qualitative properties of such an differential equation. We refer to the monograph by Arnold \cite{Arn98}.

\smallskip

We consider a metric dynamical system $(\Omega,\fF,\PP,\theta)$ where $(\Omega,\fF,\PP)$ is a probability space and $\theta$ is a $\bB(\RR)\otimes\fF,\fF$--measurable flow on $\Omega$, that is:
\begin{equation*}
 \theta_0={\rm id}_\Omega,\qquad \theta_t\theta_s=\theta_{t+s},\, t,s\in \RR.
\end{equation*}

\begin{definition}\label{grds_d1}
Let $V$ be some topological space.
A random dynamical system over the metric dynamical system $(\Omega,\fF,\PP,\theta)$ is a $\bB(\RR^+)\otimes\fF\otimes \bB(V),\bB(V)$--measurable mapping
such that the cocycle property holds
\begin{equation*}
  \phi(t,\omega,x)=\phi(t-\tau,\theta_\tau \omega,\phi(\tau,\omega,u_0)),\qquad \phi(0,\omega,u_0)=u_0,
\end{equation*}
for all $t\ge\tau\in\RR^+$, $u_0\in V$ and $\omega\in\Omega$.
\end{definition}

For our application we will choose for $V$ the Hilbert--space introduced in Section \ref{s1}. A metric dynamical system is a model for a noise and in this paper we are interested in a fractional noise in $V$ of trace class.
Given a probability space and $H\in (0,1)$, a continuous centered Gau{\ss}ian--process
$\beta^H(t)$, $t\in\mathbb{R}$, with the covariance function
\begin{equation*}
\mathbb{E}\beta^H(t)\beta^H(s)=\frac{1}{2}(|t|^{2H}+|s|^{2H}-|t-s|^{2H}),\qquad t,\,s\in\mathbb{R}
\end{equation*}
is called a {\it two--sided one-dimensional fractional Brownian--motion (fBm)}, and $H$ is the Hurst--parameter.
Let $Q$ be a positive symmetric operator of trace class on $V$, i.e.,  ${\rm tr}_V\,Q<\infty$, with positive discrete spectrum $(q_i)_{i\in\NN}$ and eigenelements $(f_i)_{i\in \NN}$.
For simplicity, we assume that $f_i=e_i$.
Then a continuous {\it $V$-valued fractional Brownian--motion $B^H$} with  covariance operator $Q$ and Hurst parameter $H$ is defined by
\begin{equation*}
   B^H(t)=\sum_{i=1}^{\infty} \sqrt{q_i}e_i \beta_i^H(t),\quad t\in\mathbb{R},
\end{equation*}
where $(\beta_i^H(t))_{i\in{\mathbb N}}$ is a sequence of stochastically independent one-dimensional fBm.
It is known that the canonical probability space $(C_0(\RR,V),\bB(C_0(\RR,V)),\PP_H,\theta)$ is a metric dynamical system where $\PP_H$ is the fractional Gau{\ss}--measure with Hurst--parameter $H\in (0,1)$ and determined by $Q$,
and $C_0(\RR,V)$ is the space of continuous paths in $V$ with value zero at zero.
We restrict  this metric dynamical system to the set $\Omega$
of $\beta^{\prime\prime}$--H{\"o}lder--continuous paths on $[-m,m]$ for any $m\in \NN$ and for $1/3<\beta^\prime<\beta^{\prime\prime}<H\le 1/2$, and equip this set with the $\sigma$--algebra $\fF:=\Omega \cap \bB(C_0(\RR,V))$
and restrict $\PP_H$ to this new domain. For $\theta$ we take the Wiener--shift given by \begin{equation*}
  \theta_t\omega(\cdot)=\omega(\cdot+t)-\omega(t),\quad t\in\RR,\,\omega\in\Omega.
\end{equation*}
The set $\Omega \in \bB(C_0(\RR,V))$, it has full measure and it is $\theta$--invariant. Moreover, the quadruple $(\Omega,\fF,\PP_H,\theta)$ is a metric dynamical system. For details we refer to  \cite{ChGGSch12}, \cite{CGSV10} or \cite{GSch11}.

Note that above, in the definition of $\Omega$, we have picked a new parameter $\beta^{\prime\prime}$ such that $1/3<\beta^\prime<\beta^{\prime\prime}<H\le 1/2$. By assumption {\bf(H3)} we can obtain a set $\Omega$ satisfying this regularity condition. The reason to introduce this $\beta^{\prime\prime}$ condition
follows from \cite{GaLuSch} and \cite{GLS14d}.

\smallskip

We would like now to apply the pathwise character of  the  integral given in Section \ref{s2} to prove the existence of a random dynamical system generated by the mild solution to (\ref{grds_eq011}).
For this purpose we have to modify {\bf (H3)}, allowing to treat equations driven by an fBm.

\medskip

Let $C(\Delta,L_2(L_2(V,V_\kappa),V\otimes V))$ be the separable Fr\'echet--space of continuous functions defined on $\Delta=\{(s,t)\in \RR^2: s\le t\}$ with values in $L_2(L_2(V,V_\kappa),V\otimes V)$.

\medskip

{\bf (H$3^\prime$)} There exists a
$\fF,\bB(C(\Delta,L_2(L_2(V,V_\kappa),V\otimes V)))$--measurable random variable $(\omega\otimes_S\omega)$ having the following properties:

\smallskip

(i) There exists a $\theta$--invariant set of full measure $\bar\Omega\in\fF$ such that for any $\tau\in\RR$ we have that
\begin{equation*}
(\omega\otimes_S\omega)(\tau+\cdot,\tau+\cdot)=(\theta_\tau\omega\otimes_S\theta_\tau\omega)(\cdot,\cdot).
\end{equation*}

\smallskip

(ii) For any $m\in\NN$ and $\omega$ contained in a full set $\Omega^\prime\subset \bar\Omega$ we have

$$\lim_{n\to \infty} (\omega_n,(\omega_n\otimes_S\omega_n))=(\omega,(\omega\otimes_S\omega)),\, \text {in }
C_{\beta^\prime}([-m,m];V)\times C_{2\beta^\prime} (\bar \Delta_{-m,m}; L_2(L_2(V,\hat V),V\otimes V)).$$

\smallskip

From now on, outside the invariant set $\Omega^\prime$ we define the elements $(\omega,\,(\omega\otimes_S\omega))\equiv 0$.

\medskip

Note that the first part of this hypothesis holds when considering $\omega$ to be a path of a fBm with $H\in (1/3,1/2]$, see \cite{GLS14d} Theorem 16. In fact, we can prove that there exists a strongly stationary version of $(\omega\otimes_S\omega)$ which is indistinguishable to the version stated in \cite{GLS14d} Theorem 1 with respect to $\bar\Omega$, that we denote by the same symbol, and such that
\begin{equation*}
   (\theta_\tau\omega\otimes_S\theta_\tau\omega)(\cdot,\cdot)=(\omega\otimes_S\omega)(\cdot+\tau,\cdot+\tau).
\end{equation*}
In particular, the $\theta$--invariance of $\bar\Omega$ follows easily by Lederer \cite{Led01} and K{\"u}mmel \cite{Kum14}.\\

By Deya {\it et al.} \cite{DeNeTi10} we have that for any $m\in \NN$ all conditions from {\bf (H3)} hold if we replace the interval $[0,T]$ by  $[-m,m]$
with respect to a $\theta$--invariant set. \\

On the other hand, the relationship of {\bf(H3$^\prime$)} (i) also holds if we replace  $\omega$ by $\omega_n$. Then on a full  set $\Omega^\prime\subset \bar\Omega$ for $\tau\in\RR$ we have that
\begin{equation*}
(\omega_n\otimes_S\omega_n)(\cdot+\tau,\cdot+\tau)-(\omega\otimes_S\omega)(\cdot+\tau,\cdot+\tau)=(\theta_\tau\omega_n\otimes_S\theta_\tau\omega_n)
  -(\theta_\tau\omega\otimes_S\theta_\tau\omega).
\end{equation*}
Hence the left side converges iff the right side converges. But the left side converges to zero on $C^{2\beta^\prime}(\bar\Delta_{-m,m})$ for any $m\in \NN$ iff
\begin{equation*}
(\omega_n\otimes_S\omega_n)(\cdot,\cdot)-(\omega\otimes_S\omega)(\cdot,\cdot)
\end{equation*}
tends to zero.
Thus we have the convergence of the right hand side on the $\theta$-invariant set $\Omega^\prime$ of full measure.

\smallskip

On the other hand, from Theorem \ref{tex2b} we know that for {\em any} $T$ there exists a unique solution $U$ with respect to the domain $[0,T],\Delta_{0,T}$. We may extend this solution to $\RR^+,\Delta_{0,\infty}$. In particular, take $T=T_n=n$, then the pair $(u(t),v(s,t))$ is given by the solution $U$ with respect to $[0,n],\Delta_{0,n}$ where $t<n$.
By the uniqueness result, the restriction of a solution from $[0,T],\Delta_{0,T}$ to $[0,T^\prime],\Delta_{0,T^\prime}$, $T^\prime<T$ is
a solution on the later domain. Hence this extension makes sense.

\begin{theorem}\label{grds_t3}
Assume {\bf (H1)-(H2)} and {\bf (H3$^\prime$)}.  Let $U=(u,v)$ be the solution from Theorem \ref{tex2b} extended to $\RR^+,\Delta_{0,\infty}$ for $\omega\in\Omega^\prime$ with initial condition
$u_0\in V$. Then $u$ generates a random dynamical system on $\RR^+\times \Omega^\prime \times V$.
\end{theorem}

\begin{proof}
Since the restriction of $U$ to $[0,T], \Delta_{0,T}$ is in $\hat W_{0,T}$, it follows that the integral $\tT_1(U)$ is additive.
Hence  $U=(u,v)$ restricted to $[\tau,T],\Delta_{\tau,T}$
solves
\begin{equation*}
 U= \tT^\tau(U,\omega,(\omega \otimes_S \omega), \tT_1(U,\omega,(\omega\otimes_S\omega),u_0)(\tau))
\end{equation*}
uniquely for $\tau>0$, $\omega\in\Omega^\prime$ and $u_0\in V$.
Setting
\begin{equation}\label{grds_eq6}
\phi(t,\omega,u_0):=\tT_1(U,\omega,(\omega\otimes_S\omega),u_0)(t)=u(t)
\end{equation}
we obtain the cocycle property if
\begin{equation}\label{grds_eq8}
\tT_1(U,\omega,(\omega\otimes_S\omega),u_0)(t)=\tT_1(U_\tau,\theta_\tau \omega,(\theta_\tau\omega\otimes_S\theta_\tau\omega),u(\tau))(t-\tau)
\end{equation}
where $U_\tau(s,t)=(u_\tau,v_\tau)(s,t)=(u(\tau+t), v(\tau+s,\tau+t))$. Note that because $u,\,v $ and $\omega$ are connected by the Chen property (\ref{chen}), the same holds for $u_\tau,\,v_\tau$ and $\theta_\tau \omega$.

For the classical evolution equation \eqref{grds_eq10}  we get
\begin{align}\label{grds_eq9}
\begin{split}
 & \tT_1((u_n(\cdot),(u_n\otimes\omega_n)(\cdot,\cdot)),\omega_n,(\omega_n\otimes_S\omega_n), u_n(0))(t)\\
&=\tT_1^\tau((u_n(\cdot),(u_n\otimes\omega_n)(\cdot,\cdot)),\omega_n,(\omega_n\otimes_S\omega_n),u_n(\tau))(t-\tau)\\
&=\tT_1((u_n(\cdot+\tau),(u_n\otimes\omega_n)(\cdot+\tau,\cdot+\tau)),\theta_\tau\omega_n,(\theta_\tau\omega_n\otimes_S\theta_\tau\omega_n),u_n(\tau))(t-\tau)\\
&=\tT_1((u_n(\cdot+\tau),(u_n(\cdot+\tau)\otimes\theta_\tau\omega_n)(\cdot,\cdot)),\theta_\tau\omega_n,(\theta_\tau\omega_n\otimes_S\theta_\tau\omega_n),u_n(\tau))(t-\tau)
\end{split}
\end{align}
for every $\omega\in\Omega^\prime$ because
\begin{equation*}
 (\omega_n)^\prime(\cdot+\tau)=(\theta_\tau\omega_n)^\prime(\cdot).
\end{equation*}

Note that in the last formulas we could omit $(\theta_\tau\omega_n\otimes_S\theta_\tau\omega_n)$ because the classical solutions are independent of this term. Moreover, above we have used that the $\tT_2$-component given by $(u_n\otimes\omega_n)$ (see \eqref{grds_eq3b}) has the property
\begin{equation*}
(u_n\otimes\omega_n)(\cdot+\tau,\cdot+\tau)=(u_n(\cdot+\tau)\otimes\theta_\tau\omega_n)(\cdot,\cdot).
\end{equation*}

Therefore, since $U\in \hat W_{0,T}$, the convergence conclusion of Theorem \ref{tex2b}, which holds for every $\omega\in\Omega^\prime$ independently of $u_0$, applied to the left and right side of \eqref{grds_eq9} yields \eqref{grds_eq8}, and therefore the cocycle property is established for $\phi$.

\smallskip

Now we deal with the measurability of $\phi$. Consider \eqref{grds_eq6} replacing $(\omega,(\omega\otimes_S\omega))$ by $(\omega_n,(\omega_n\otimes_S\omega_n))$.
The mapping $\omega\mapsto\omega_n$ is $\fF,\fF$-measurable for every $n\in\NN$ and hence the mapping $\omega\mapsto (\omega_n\otimes_S\omega_n)$ is
$\fF, \bB(C(\Delta,L_2(L_2(V,V_\kappa),V\otimes V))$--measurable. Let $U_n=(u_n,v_n)$ be the solution for parameters $(\omega_n,(\omega_n\otimes_S\omega_n))$. Then
$u_n$ is the classical solution to \eqref{grds_eq011} (and $v_n=(u_n\otimes\omega_n)$). Hence for the first component of $U_n$ we have that
\begin{equation*}
[0,T]\times \Omega^\prime \times V\ni  (t,\omega,u_0)\mapsto u_n(t)\in V
\end{equation*}
is measurable with respect to $\bB([0,T])\otimes\fF\otimes\bB(V),\bB(V)$. However from (\ref{conv}) that
\begin{equation*}
  \hat W_{0,T}-\lim_{n\to\infty} U_n=U
\end{equation*}
and hence for the first components of $U$ and $U_n$ we have for every $t\in [0,T]$
\begin{equation*}
  \lim_{n\to\infty} u_n(t)=u(t)
\end{equation*}
which gives the measurability of $u(t)$ with respect to $\fF\otimes \bB(V)$. By the fact that $u(t)$ is continuous in $t$, see Castaing and Valadier, Chapter III \cite{CasVal77}, we obtain by \eqref{grds_eq6} the $\bB([0,T])\otimes\fF\otimes\bB(V),\bB(V)$--measurability, that is, measurability restricted to $[0,T]$. Considering an increasing
sequence $(T_n)_{\in \NN},\,T_n>0$ with limit $\infty$ and setting
\begin{equation*}
  u_n(t)=\left\{
  \begin{array}{lcr}
  u(t)&:& t\le T_n\\
  u(T_n)&:& t\ge T_n,
  \end{array}
  \right.
\end{equation*}
then $u_n$ is $\bB(\RR^+)\otimes\fF\otimes\bB(V),\bB(V)$--measurable and for any $(t,\omega,u_0)$ the mapping $u_n$ converges to the first component
$u$ of the extended solution $U$. Hence $\phi$ is $\bB(\RR^+)\otimes\fF\otimes\bB(V),\bB(V)$--measurable.

\end{proof}

\begin{remark}
There is another way to prove that the path--area solution generates a random dynamical system which is based on Kager and Scheutzow \cite{KagScheu97}.
In particular \eqref{grds_eq9} gives us the random flow property for the $u$--component of the path--area solution. The conditions of \cite{KagScheu97} Theorem 4 to generate a cocycle from a (semi)flow are fulfilled. The measurability condition {\bf (iii')} in that article follows (for instance) because $\tT_1^\tau$ is a pointwise limit of the classical solution having these measurability properties. {\bf (v')} is due to the continuity of $t\mapsto \tT_1^\tau(t)$ and {\bf (iv')} follows from the additivity of the pathwise integral and the continuous dependence of $\tT$ on the initial condition.
\end{remark}

\section{Example}

Let us assume that $A$ is generated by the Laplacian on $\mathcal O =(0,1)$ with  homogenous Dirichlet boundary condition. $A$ with domain $D(-A)=H^2(\oO)\cap H_0^1(\oO)$ generates a semigroup in $L_2(\oO)$ .
Let $\rho=1/4+\eps$, $\eps>0$, small. Then $V:=D((-A)^\rho)$ consists of the Slobodetski spaces $H^{2\rho}(\mathcal O)$
satisfying the homogeneous boundary conditions, see Da Prato and Zabczyk \cite{DaPrato}, Page 401. In particular, the continuous embedding $V\subset C(\bar\oO)$ holds. In what follows we consider the restriction of the semigroup $S$ to the space $V$. Note that the inequalities \eqref{eq1} and \eqref{eq2} continue being true, and that $(\lambda_i^{-\rho}e_i)_{i\in \NN}$ is an    orthonormal basis of $V$ where  $(\lambda_i)_{i\in\NN}$
is the spectrum of $A$ and $(e_i)_{i\in \NN}$ are the associated eigenelements of $A$ with respect to $L_2(\oO)$, which are uniformly bounded in $L_\infty(\mathcal O)$.
The asymptotical behavior of the spectrum is given by $\lambda_i\sim i^2$.

\begin{lemma}
For $\mu\in (1,5/4)$
\begin{equation*}
  D((-A)^{\mu})= H^{2\mu}(\mathcal O)\cap H_0^1(\oO).
\end{equation*}
\end{lemma}
\begin{proof}
On $H^{2\mu}(\mathcal O)\cap H_0^1(\oO)$ we know that $A=\Delta_{HDBC}$  which is an isomorphism with range $H^{2\mu-2}(\mathcal O)$,
see Egorov and Shubin \cite{egorov}, Page 124. In addition, $(-A)^{\mu-1}$ has the domain $H^{2\mu-2}(\mathcal O)$ if $\mu\in (1,5/4)$,  see  \cite{DaPrato}, Page 401.
\end{proof}
Now for $1/3<\beta^\prime<1/2$ we take
\begin{equation*}
  \gamma>1-\beta^\prime\,\qquad\kappa=\gamma+\rho< 1.
\end{equation*}
That choice of $\gamma$ and $\kappa$ ensures that we can define $(\omega\otimes_S\omega)$ as in \cite{GLS14d} Theorem 1 (see also Lemma 3 of that paper). In addition, we can ensure the assumptions of the global existence Theorem 18 since the above choice implies $\kappa+\beta^\prime>1$.

Let $g$ be a four times differentiable  function on $\bar{\mathcal O}\times \RR$  which is zero on $\{0,1\}\times \RR$, such that all the corresponding derivatives ($g$ itself  included) are bounded. Define

\begin{equation*}
    G(u)(v)[x]=\int_{\mathcal O} g(x,u(y))v(y)dy\quad \text{for }u,\,v\in V.
\end{equation*}
Following Kantorowitsch and Akilow \cite{KA} Section XVII.3  it is not hard to prove that  $G$ is  three times continuously differentiable where the derivatives are given by
\begin{align*}
    DG(u)(v,h_1)[x]&=\int_{\mathcal O}D_2g(x,u)v(y)h_1(y)dy,\\
    D^2G(u)(v,h_1,h_2)[x]&=\int_{\mathcal O}D_2^2g(x,u)v(y)h_1(y)h_2(y)dy,\\
     D^3G(u)(v,h_1,h_2,h_3)[x]&=\int_{\mathcal O}D_2^3g(x,u)v(y)h_1(y)h_2(y)h_3(y)dy,
\end{align*}
for $v,\,h_1,\,h_2,\,h_3\in V$. For $\mu\in (1,5/4)$ we obtain that  $G(u)(v),\, DG(u)(v,h_1)$, $ D^2G(u)(v,h_1,h_2),$ $\,D^3G(u)(v,h_1,h_2,h_3) \in H^3(\oO)\cap H^1_0(\oO)\subset D((-A)^{\mu})\subset V_\kappa:=$ $D((-A)^{\rho+\kappa})= D((-A)^{2\rho+\gamma})$ since with the choice we have done we have $2\rho+\gamma \in (1,5/4)$. Let us check, for instance, that $D^3G(u)(v,h_1,h_2,h_3) \in D((-A)^\mu)\subset V_\kappa$. By the
continuous embedding theorem we have that
\begin{align*}
     \int_{\mathcal O}&\bigg|\int_{\mathcal O}D^2_2D_1^kg(x,u(y)+h_3(y))v(y)h_1(y)h_2(y)-D^2_2D_1^kg(x,u(y))v(y)h_1(y)h_2(y)\\
&-D_2D_2^2 D_1^k g(x,u(y))v(y)h_1(y)h_2(y)h_3(y)dy\bigg|^2dx\le c\bigg(\int_{\mathcal O}|v(y)h_1(y)h_2(y)||h_3(y)|^2dy\bigg)^2\\
&\le c^\prime|v|_{C}^2|h_1|_{C}^2|h_2|_{C}^2|h_3|_{C}^4\le c^{\prime\prime} |v|^2|h_1|^2|h_2|^2|h_3|^4\quad\text{for } k=1,2,3,
\end{align*}
where $c^\prime$ is a uniform bound for $|D_2^4D_1^kg(x,u)|^2|\oO|$.
This gives the differentiability of $D^2G(u)$ in $H^3(\oO)\cap H_0^1(\oO)$ and since
$G(u)(v)[x],\cdots,D^3G(u)(v,h_1,h_2,h_3)[x]$ are zero for $x\in\{0,1\}$ we have differentiability in $D((-A)^\mu)$ too.

The Hilbert-Schmidt property of $DG(u)$ follows by
\begin{equation*}
    \sup_{k=1,2,3} \sum_{i,j}\int_{\mathcal O}\bigg(\int_{\mathcal O}|D_2D_1^k g(x,u(y))\lambda_i^{-\rho}e_{i}(y)\lambda_j^{-\rho}e_{j}(y)dy\bigg)^2dx<c(\sum_i \lambda_i^{-2\rho})^2<\infty,
\end{equation*}
due to the  boundedness of $D_2 D_1^k g$ and the uniform boundedness of $(e_i)_{i\in\NN}$ in $L_\infty(\oO)$. In the same manner we obtain that the other derivatives are Hilbert--Schmidt operators.

\medskip

These estimates allow us to apply Theorem \ref{tex2b} to SEEs that have the above kernel integral diffusion operator. For a different example of diffusion we refer the reader to \cite{GaLuSch}.

\end{document}